\documentclass[12pt, reqno]{amsart}
\usepackage{amsfonts}
\usepackage{amssymb}
\usepackage{latexsym}

\usepackage{tikz}
\usepackage{graphicx,graphics,epsfig}
\usepackage{epstopdf}	
\usepackage{color}
\usepackage{a4wide}

\usepackage[]{enumerate}
\usepackage{verbatim}
\usepackage{bbm}

\usepackage[]{enumerate}

\usepackage[sc]{mathpazo}          
\usepackage{eulervm}               
\usepackage[scaled=0.86]{berasans} 
\usepackage[scaled=1]{inconsolata} 
\usepackage[T1]{fontenc}

\usepackage[final]{hyperref}
\hypersetup{colorlinks=true, linkcolor=blue, anchorcolor=blue, citecolor=red, filecolor=blue, menucolor=blue, pagecolor=blue, urlcolor=blue}

\newcommand{\Bigabs}[1]{\Bigl\vert #1 \Bigr\vert}

\newcommand{\norm}[1]{\left\Vert #1 \right\Vert}

\newcommand{\Z}{\mathbb{Z}}
\newcommand{\R}{\mathbb{R}}

\newcommand{\angles}[1]{\langle #1 \rangle}

\DeclareMathOperator{\im}{Im}
\DeclareMathOperator{\re}{Re}

\DeclareMathOperator{\sgn}{sgn}
\DeclareMathOperator{\supp}{supp}

\newtheorem{theorem}{Theorem}[section]
\newtheorem{proposition}[theorem]{Proposition}
\newtheorem{remark}[theorem]{Remark}

\newtheorem{lemma}[theorem]{Lemma}

\numberwithin{equation}{section}

\title[Dispersive estimates for full dispersion KP equations]{Dispersive estimates for full dispersion KP equations}

\author[D. Pilod]{Didier Pilod}
\address{Department of Mathematics, University of Bergen, Postbox 7800, 5020 Bergen, Norway}
\email{Didier.Pilod@iuib.no}

\author[J.-C. Saut]{Jean-Claude Saut}
\address{Laboratoire de Math\' ematiques, UMR 8628\\
Univ. Paris-Sud, CNRS, Universit\' e Paris-Saclay \\91405 Orsay, France}
\email{jean-claude.saut@universite-paris-saclay.fr}

\author[S. Selberg]{Sigmund Selberg}
\address{Department of Mathematics, University of Bergen, Postbox 7800, 5020 Bergen, Norway}
\email{Sigmund.Selberg@uib.no}

\author[A. Tesfahun]{Achenef Tesfahun}
\address{Department of Mathematics, University of Bergen, Postbox 7800, 5020 Bergen, Norway}
\email{achenef@gmail.com}


\begin{document}

\begin{abstract} 
We prove several dispersive estimates for the linear part of the Full Dispersion Kadomtsev-Petviashvili introduced by David Lannes to overcome some shortcomings of the classical Kadomtsev-Petviashvili equations. The proof of these estimates combines the stationary phase method with sharp asymptotics on \emph{asymmetric Bessel functions}, which may be of independent interest.  As a consequence, we prove that the initial value problem associated to the Full Dispersion Kadomtsev-Petviashvili  is locally well-posed in $H^s(\mathbb R^2)$, for $s>\frac74$, in the capillary-gravity setting. 
\end{abstract}

\maketitle
\section{Introduction}

\subsection{Introduction of the model and physical motivation}

The classical Kadomtsev-Petviashvili equation (KP)
\begin{equation}\label{KP}
\partial_tu+\partial_{x_1}u+u\partial_{x_1}u+\partial_{x_1}^3u\pm \partial_{x_1}^{-1}\partial_{x_2}^2u=0,
\end{equation}
where $+$ corresponds to KP-II and $-$ to KP-I, was introduced in  the pioneering paper \cite{KaPe} in order to investigate the  stability properties of the KdV soliton with respect to long wave perturbations in the transverse direction. We are here in a long wave regime, that is the wavelengths in $x_1$ and $x_2$ are large, those in $x_2$ being larger.

Actually the derivation in \cite {KaPe} was formal and concerned only the linear transport part of equation \eqref{KP}, in particular it is independent of the dispersive and nonlinear terms.
It is only related  to the finite propagation speed properties of the transport operator
$M=\partial_{t}+\partial_{x_1}$. 

Recall that $M$ gives rise to one-directional waves moving to the right with speed one;
i.e., a profile $\varphi(x_1)$ evolves under the flow of $M$ as $\varphi(x_1-t)$.
A weak transverse perturbation of $\varphi(x_1)$ 
is a two-dimensional function $\psi(x_1,x_2)$ close to $\varphi(x_1)$, localised in the frequency region
$\big|\frac{\xi_2}{\xi_1}\big|\ll 1$, where $\xi_1$ and $\xi_2$ are the Fourier modes corresponding to $x_1$ and $x_2$, respectively.
We look for a two-dimensional perturbation  
\[\widetilde{M}=\partial_{t}+\partial_{x_1}+\omega(D_{1},D_{2}) \] of $M$ such that,
similarly to above, the profile of $\psi(x_1,x_2)$ does not change much
when evolving under the flow of $\widetilde{M}$. Here $\omega(D_{1},D_{2})$ denotes the Fourier multiplier with symbol
the real function $\omega(\xi_1,\xi_2)$. 
Natural generalizations of the flow of $M$ in two dimensions are the flows of the wave operators 
$\partial_{t}\pm\sqrt{-\Delta}$ which enjoy the finite propagation speed
property. 
Since
\[
\sqrt{\xi_1^2+\xi_2^2} \sim \pm\left(\xi_1+\frac{1}{2}\xi_1^{-1}\xi_2^{2}\right) , \quad  {\rm when }\quad |\xi_1|, \Big|\frac{\xi_2}{\xi_1}\Big|\ll 1,
\]
we deduce the approximation in this regime
\[
\partial_{t}+\partial_{x_1}+\frac{1}{2}\partial_{x_1}^{-1}\partial_{x_2}^{2}\sim
\partial_{t}\pm \sqrt{-\Delta},
\]
which leads to the correction $\omega(D_1,D_2)=\frac{1}{2}\partial_{x_1}^{-1}\partial_{x_2}^{2}$.

Of course when the transverse effects are two-dimensional, the correction is $\frac{1}{2}\partial_{x_1}^{-1}\Delta_{\perp}$,
where $\Delta_{\perp}=\partial_{x_2}^{2}+\partial_{x_3}^{2}$.

Note that the term $\frac{1}{2}\partial_{x_1}^{-1}\partial_{x_2}^{2}$ leads to a singularity at $\xi_1=0$ in Fourier space which is not present in the original physical context where the KdV equation was derived and there is a price to pay for that, various  shortcomings of the KP equation that we will describe now.

The first one concerns the accuracy of the KP approximation as a water wave model.
As aforementioned, the derivation in \cite{KaPe} did not refer to a specific physical content. Its formal derivation in the context of water waves was done in \cite{AS} but a rigorous derivation, including error estimates, was only achieved in \cite{L}. It is shown there that the error estimate between solutions of the full water waves system and solutions of the KP-II equation has the form, in suitable Sobolev norms, and for a fixed time interval,
$$||U_{WW}-U_{KP}||=o(1),$$
while the corresponding error in the Boussinesq (KdV) regime  is $O(\epsilon^2t)$ where the small parameter $\epsilon$ measures the comparable effects  of shallowness and nonlinearity.

Another shortcoming of the KP equation is the (unphysical) constraint implied by the $\partial_{x_1}^{-1}\partial_{x_2}u$ term. Actually, in order to make sense, $u$ should satisfy the constraint
$\hat u(0,\xi_2)=0,\; \forall \xi_2\in \R,$ or alternatively $\int_{\R}u(x_1,x_2)dx_1=0,\; \forall x_2\in \R,$ that makes no sense for real waves. We refer to \cite{MST1} for further comments and results on this \lq\lq constraint problem.\rq\rq

Another drawback is a singularity at time $t=0$ that is already present at the linear level. Denoting by $S_{\pm}(t)$ the (unitary in $L^2$)  linear group of the KP equations, one can express the linear solution corresponding to any $L^2$ initial data $u_0$ (without any constraint) in Fourier variables by
$$S_{KP\pm}(t)\widehat{u_0}(\xi_1,\xi_2)=\hat u(\xi_1,\xi_2,t)=\exp\left\lbrace it\left(\xi_1^3\pm \frac{\xi_2^2}{\xi_1}\right)\right\rbrace\widehat{u_0}(\xi_1,\xi_2),$$
which defines of course a unitary group  in any Sobolev space $H^s(\R^2)$. 
On the other hand, even for smooth initial data, say in the Schwartz class, the relation
$$u_{x_1t}=u_{tx_1}$$
holds true only in a very weak sense, {\it e.g.} in $\mathcal S'(\R^2),$ if $u_0$ does not satisfy the constraint $\hat u_0(0,\xi_2)=0$ for any $\xi_2 \in \R$, or equivalently $\int_{-\infty}^{\infty} u_0(x_1,x_2)dx_2 =0$ for any $\xi_2\in \R$.

In particular, even for smooth localised $u_0,$ the mapping 
$$\hat{u}_0\mapsto \partial_t \hat{u}=i\left(\xi_1^3\pm \frac{\xi_2^2}{\xi_1}\right)\exp\left\lbrace it\left(\xi_1^3\pm \frac{\xi_2^2}{\xi_1}\right)\right\rbrace \hat{u}_0(\xi)$$
cannot be defined with values in a Sobolev space if $u_0$ does not satisfy the zero mass constraint.
In particular, if $u_0$ is a gaussian, $\partial_t u$ is not even in $L^2.$



\smallskip
Those shortcomings have led David Lannes \cite{La} to introduce in the KP regime a full dispersion counterpart of the KP equation that would not suffer of such defects or at least at a lower level,\footnote{We refer to \cite {LS2} for another approach for an asymptotic water model in the KP regime leading to a local {\it weakly transverse Boussinesq system} leading to the optimal error estimate with the solutions of the full water waves system.}.

\vspace{0.3cm}
This full dispersion KP equation (FDKP) reads
\begin{equation}\label{FDKP}
\partial_tu+ L_{\beta , \epsilon} (D) \left(   1+ \epsilon \frac{D_2^2}{D_1^2}\right)^\frac12 \partial_{x_1} u+ 3\epsilon \partial_{x_1}( u^2)=0  \, ,
\end{equation}
where $u=u(x_1,x_2,t)$ is a real-valued function, $(D_1, D_2)=(-i\partial_{x_1}, -i\partial_{x_2})$,  $D^\epsilon= (D_1, \sqrt{\epsilon}D_2)$, hence $$ |D^\epsilon| =\sqrt{D_1^2 +\epsilon D_2^2 },$$ and $ L_{\beta, \epsilon} $ is a non--local operator defined by
$$L_{\beta, \epsilon} (D)= \left(1+  \beta  \epsilon |D^\epsilon|^2\right)^\frac 12 \left(\frac{\tanh( \sqrt \epsilon  |D^\epsilon|)}{\sqrt \epsilon  |D^\epsilon|}\right)^{1/2}.$$
Here $\beta\ge 0$ is a dimensionless coefficient measuring the surface tension
    effects and $\epsilon >0$ is the shallowness parameter which is proportional to the ratio of the amplitude of the wave to the mean depth of the fluid.
    
In the case of purely gravity waves ($\beta=0$), the symbol of \eqref{FDKP}  writes 
\begin{equation}\label{symbol:p}
p(\xi_1,\xi_2)=\frac{i}{ \epsilon^{1/4}} \left(\tanh [ \sqrt\epsilon (\xi_1^2+\epsilon \xi_2^2)^{\frac{1}{2}}]\right)^{\frac{1}{2}} (\xi_1^2+\epsilon \xi_2^2)^{\frac{1}{4}} \text{sgn}\;\xi_1,
\end{equation}
while in the case of gravity-capillary waves ($\beta>0$), the symbol is
\begin{equation} \label{symbol:pbeta}
\tilde p(\xi_1,\xi_2)= \left(1+\beta\epsilon(\xi_1^2+\epsilon \xi_2^2)\right)^{1/2}p(\xi_1,\xi_2) \, .
\end{equation}
The symbols $p$ and $\tilde p$ being real, it is clear that the linearized equations define unitary groups in all Sobolev spaces $H^s(\R^2), s\in \R.$
    
Contrary to the KP case,  $p$ and $\tilde p$  are locally bounded on $\R^2.$ However they are not continuous on the line $\lbrace (0,\xi_2), \xi_2\neq 0\rbrace,$ but they do not have the singularity $\frac{i}{\xi_1}$ of the KP equations symbols.

We now describe some links between the FDKP equation and related nonlocal dispersive equations.
We first observe that for waves depending only on $x_1$, the FDKP equation when $\beta =0$  reduces to the so-called Whitham equation (\cite{KLPS}
\begin{equation}\label{Whit_1d}
\partial_t u+\left(\frac{\tanh (\sqrt \epsilon |D_1|)}{\sqrt \epsilon |D_1|}\right)^{1/2}\partial_{x_1}u+\epsilon \frac{3}{2} u\partial_{x_1}u=0,
\end{equation}
and when $\beta>0$, it reduces to the  Whitham equation with surface tension
\begin{equation}\label{Whitbis_1d}
\partial_t u+(1+\beta \epsilon D_1^2)^{1/2}\left(\frac{\tanh (\sqrt \epsilon|D_1|)}{\sqrt \epsilon |D_1|}\right)^{1/2}\partial_{x_1}u+\epsilon \frac{3}{2} u\partial_{x_1}u=0.
\end{equation}

Note that the Whitham equations can be seen for large frequencies as perturbations of the fractional KdV (fKdV) equations 
\begin{equation}\label{fKdV1}
\partial_t u+\partial_{x_1}u+\beta^{1/2}\epsilon^{1/4}|D_1|^{1/2}\partial_{x_1}u+\epsilon \frac{3}{2}u\partial_{x_1}u=0 \, ,
\end{equation}
when $\beta>0$, and 
\begin{equation}\label{fKdV2}
\partial_t u+\partial_{x_1}u+\epsilon^{-1/4}|D_1|^{-1/2}\partial_{x_1}u+\epsilon \frac{3}{2}u\partial_{x_1}u=0 \, ,
\end{equation}
when $\beta=0$.

The FDKP equation may be therefore seen as a natural (weakly transverse)  two-dimensional version of the Whitham equation, with and without surface tension.

On the other hand, those fKdV equations  have KP versions, namely the fractional KP equations (fKP), see \cite{LPS1}, which write for a general nonlocal operator  $D_1^\alpha$ :
\begin{equation}\label{fKP}
\partial_t u+u\partial_{x_1}u-D_1^\alpha\partial_{x_1}u\pm \partial_{x_1}^{-1}\partial_{x_2}^2 u=0,\quad -1<\alpha <1
\end{equation}
in its two versions, the fKP-II ($+$ sign) and the fKP-I version ($-$ sign).

The fKP equation has several motivations. First, when $\alpha =1$ the fKP-II equation is the relevant version of the Benjamin-Ono equation. For general values of $\alpha$ the fKP equation is the KP version of the fractional KdV equation (fKdV), which in turn is a useful toy model to understand the effect of a \lq\lq weak\rq\rq \, dispersion on the dynamics of the inviscid Burgers equation. When $-1<\alpha<0$, both equations are mainly \lq\lq hyperbolic\rq\rq, with the possibility of shocks (but a dispersive effect leading to the possibility of global existence and scattering of small solutions) while when $0<\alpha<1$ the dispersive effects are strong enough to prevent the appearance of shocks for instance. We refer to \cite{LPS2, LPS1, KLPS, KS} for results and numerical simulations on those equations. 

The fKP equation is also the KP version of the (inviscid) Khokhlov-Zabolotskaya-Kuznetsov (KZK) equation (see \cite{Roz}) which has a \lq\lq hyperbolic\rq\rq \, character with the possible appearance of shocks.

Those fKP equations are not directly connected to the FDKP equation but they share the property of being two dimensional nonlocal dispersive perturbations of the Burgers equation.

\smallskip
 Some of the properties of the FDKP equations are  displayed in \cite{LS2}.  In particular, it is easy to check by viewing it as a skew-adjoint perturbation of the Burgers equation, that the Cauchy problem is locally well-posed, without need of any constraint, in $H^s(\R^2), s>2.$ Note that this result does not use any dispersive property of the linear group. 
    
 The natural energy space associated to the FDKP equation is, in the case without surface tension $\beta=0$,
 $$E=\left\lbrace u\in L^2(\R^2)\cap L^3(\R^2) : |D^\epsilon|^{1/4} |D_1|^{-1/2} u, |D^\epsilon|^{1/2}|D_1|^{-1/2} u\in L^2(\R^2)\right\rbrace  .$$
This space is associated to a natural Hamiltonian. In fact, as for the classical KP I/II equations, the $L^2$ norm is formally conserved by the flow of  \eqref{FDKP}, and so is the Hamiltonian 
\begin{equation}\label{Hamil}
\mathfrak H_\epsilon(u)=\frac{1}{2}\int_{\R^2}|H_\epsilon(D) u|^2+\frac{\epsilon}{4}\int_{\R^2} u^3,
\end{equation}
where
\begin{equation}\begin{split}
H_\epsilon(D)&=\left(\frac{(1+\sigma\epsilon |D^\epsilon|^2)\tanh (\sqrt \epsilon |D^\epsilon|)}{\epsilon^{1/2}|D_\epsilon|}\right)^{1/4}\left(1+\epsilon \frac{D_2^2}{D^2_1}\right)^{1/4}\\ 
 &=\left(\frac{(1+\sigma\epsilon |D^\epsilon|^2)\tanh (\sqrt \epsilon |D^\epsilon|)}{\epsilon^{1/2}}\right)^{1/4}\frac{|D^\epsilon|^{1/4}}{|D_1|^{1/2}} \, .
 \end{split}\end{equation} 
The cases $\beta=0$ and $\beta>0$ correspond respectively to purely gravity waves and capillary-gravity waves.

One finds the standard KP I/II Hamiltonians by expanding   formally $H_\epsilon(D)$ in powers of $\epsilon,$ namely
$$H_\epsilon(D)(u)= \frac{\epsilon}{4}\int_{\R^2}\lbrack|\partial_{x_2}\partial_{x_1}^{-1}u|^2+(\beta-\frac{1}{3}) |\partial_{x_1}u|^2+ u^3\rbrack dx_1dx_2+o(\epsilon).$$

Contrary to the Cauchy problem which can be solved without constraint,  the Hamiltonian for the FDKP equation  is well defined (and conserved by the flow) provided $u$ satisfies a constraint, weaker however than that of the classical KP equations.\footnote{In the sense that the order of vanishing of the Fourier transform at the frequency $\xi_1=0$ is weaker than the corresponding one for the KP equations.}

Finally, as noticed in \cite{KS} by considering the solution of the linear KP I/II equations,
$$\hat{u}(\xi_1,\xi_2,t)=\hat{u_0}(\xi_1,\xi_2)\exp \left(it\left(\xi_1^3\pm \frac{\xi_2^2}{\xi_1}\right)\right),$$
the singularity $\frac{\xi_2^2}{\xi_1}$ implies that a strong decay of the initial data is not preserved by the linear flow,\footnote {and also actually for the nonlinear flow.} for instance the solution corresponding to a gaussian initial data cannot decay faster than $1/({x_1}^2+{x_2}^2)$ at infinity. In fact, the Riemann-Lebesgue theorem implies that $u(\cdot,t)\notin L^1(\R^2)$ for any $t\neq 0.$ The same conclusion holds of course even if $u_0$ satisfies the zero-mass constraint, {\it e.g.} $u_0\in \partial_x \mathcal S(\R^2)$ and also for the nonlinear problem as shows the Duhamel representation of the solution, see \cite{KS}.

\vspace{0.3cm}
A similar obstruction  holds for the  FDKP equations. In particular, the localised solitary waves solutions found in \cite{EG} cannot decay fast at infinity.

\subsection{Presentation of the results}
 In order to study the Cauchy problem associated to FDKP in spaces larger than the \lq\lq hyperbolic space\rq\rq \, $H^s(\R^2)$, $s>2$, and to investigate the scattering of small solutions, we will focus in this paper on the derivation of dispersive  estimates on the linear group. This is not a simple matter since the symbol  $p$ and $\tilde{p}$ defined in \eqref{symbol:p} and \eqref{symbol:pbeta} are non-homogeneous and also non-polynomial. Similar difficulties occur for other non-standard dispersive equation such as the Novikov-Veselov equation \cite{KM} or a higher dimensional version of the Benjamin-Ono equation \cite{HLRRW}. 
 
 In the rest of the paper, we work with $\epsilon=1$. Based on the identity
$$
 \left(   1+\frac{D_2^2}{D_1^2}\right)^\frac12 \partial_{x_1} = \frac{iD_1}{|D_1|}  |D| \, ,
$$
we rewrite \eqref{FDKP} as
\begin{equation}\label{FDKP1}
\partial_tu+ \widetilde L_{\beta} (D)  u+ 3 \partial_{x_1}( u^2)=0 \,  ,
\end{equation}
where 
$$ \widetilde L_{\beta} (D)=\frac{iD_1}{|D_1|}  |D| \left(1+  \beta  |D|^2\right)^\frac 12 \left(\frac{\tanh( \sqrt   |D^|)}{\sqrt   |D|}\right)^{1/2}  .$$

The solution propagator for the linear equation is given by 
\begin{align} \label{lin:KPFD}
\left[ S_{m_{\beta}} (t) f \right](x) &:= \int_{\R^2} e^{ix \cdot \xi+it \sgn(\xi_1) m_{\beta} \left(|\xi |\right) }  \hat f(\xi)   \, d\xi \,  ,
\end{align}
where 
\begin{equation} \label{def:m_beta}
m_{\beta} (r)=r \left(1+  \beta   r^2\right)^\frac 12 \left( \frac{\tanh(r)}{ r}\right)^\frac12 
\end{equation} 
and $|\xi|=\sqrt{\xi_1^2+ \xi_2^2}$.

Our first result is a $L^1-L^{\infty}$ decay estimate for the linear propagator associated to \eqref{FDKP}. Since the symbol $m_{\beta}$ is non-homogeneous, we will derive our estimate for frequency localised functions. For a dyadic number $\Lambda\in 2^\Z$, let $P_{\Lambda}$ denote the Littlewood-Paley projector localising the frequency around the dyadic number $\Lambda$ (a more precise definition of $P_{\Lambda}$ will be given in the notations below).

\begin{theorem}[Localised dispersive estimate] \label{lm-dispest}
Let $\beta \in \{ 0, 1\}$. Then, there exists a positive constant $c_{\beta}$ such that
\begin{align}
\label{dispest}
\| S_{m_{\beta}}( t)P_{\Lambda} f  \|_{L^\infty_x(\R^2)} &\le c_{\beta}  \angles{ \sqrt{  \beta} \Lambda }^{-1} \angles{ \Lambda }^{\frac32}  |t|^{-1}  \|P_{\Lambda} f\|_{L_x^1(\R^2)} 
\end{align}
for all $\Lambda\in 2^\Z$ and $f \in \mathcal{S}(\R^{2})$, and where $\angles{\xi}:= \left(1 +|\xi|^2\right)^\frac12.$
 \end{theorem}
 
 By a standard argument, the proof reduces to proving a uniform bound for the two dimensional  oscillatory integral
 \begin{equation} \label{def:I}
 I_{\Lambda, t} (x)= \int_{\R^2} e^{i  x \cdot \xi+it \sgn (\xi_1) m_{\beta} \left(|\xi |\right) }  \rho(\Lambda^{-1}|\xi|) \, d\xi \, ,
 \end{equation}
 where $\rho$ is a smooth function whose compact support is localised around $1$.
 Observe that in the KP and fractional KP cases, the corresponding oscillatory integral
$$\int_{\R^2} e^{it\left(\varphi(\xi_1)+ \frac{\xi_2^2}{\xi_1}\right)+ix\cdot \xi}d\xi \, $$
 can be reduced to a one-dimensional integral by integrating in $\xi_2$ and using the explicit representation of the linear Schr\"odinger propagator (see \cite{S, MST1,LPS2}). 

This is not the case anymore for the oscillatory integral \eqref{def:I} and for this reason we need to employ 2-dimensional methods. After passing to polar coordinates, we write
\begin{equation}\label{I-lamb}
  I_{\Lambda, t} (x)=\Lambda^2  \int_0^\infty\left[ e^{it m_\beta( \Lambda r)}J_+(\Lambda rx) +  e^{-itm_\beta( \Lambda r)}J_-( \Lambda r x)  \right] r \rho(r)   \, dr \, ,
  \end{equation}
where
\begin{equation}\label{asymbesselfunc}
J_\pm(x) = \int_{\omega \in S^1} \mathbb {1}_{\{\pm \omega_1>0\}}  e^{i   x \cdot \omega}  \, d\sigma(\omega) 
\end{equation}
are \emph{asymmetric Bessel functions}. Then, by using complex integration, we derive sharp asymptotics for these asymmetric Bessel functions which may be of independent interest (see Proposition \ref{propJ+}). With these asymptotics in hand, we can conclude the proof of Theorem \ref{lm-dispest} by combining the stationary phase method  with careful estimates on the symbol $m_{\beta}$ and its derivatives both in the case $\beta=0$ and $\beta=1$.

Once Theorem \ref{lm-dispest} is proved, the corresponding Strichartz estimates are deduced from a classical $TT^{\star}$ argument.
\begin{theorem}[Localised Strichartz estimates]\label{lm-LocStr}
Let $\beta \in \{ 0, 1\}$.  Assume that $q, r$ satisfy
  \begin{equation} \label{admissible}
 2<q \le \infty, \ 2 \le r < \infty \quad \text{and} \quad \frac1r+\frac1q=\frac12 \, .
 \end{equation} 
 Then, there exists a positive constant $c_{\beta}$ such that
 \begin{align}
\label{Str2d}
\norm{S_{m_\beta}(t) P_{\Lambda}f }_{ L^{q}_{t} L^{r}_{ x} (\R^{2+1}) } \le c_{\beta} \left[ \angles{ \sqrt{  \beta} \Lambda }^{-1} \angles{ \Lambda }^{\frac32} \right]^{\frac 12- \frac1r }
\norm{ P_{\Lambda}f}_{ L^2_{ x}(\R^2 )} \, 
\end{align}
for all $\Lambda\in 2^\Z$ and all $f \in \mathcal{S}(\R^{2})$, and where $\angles{\xi}:= \left(1 +|\xi|^2\right)^\frac12.$
\end{theorem}

\begin{remark}
Estimate \eqref{Str2d} in the case $\beta=1$ only requires a loss slightly smaller than $1/4$ derivative close to the end point $(q,r)=(2,\infty)$. This is better than the corresponding\footnote{For high frequency, the dispersive symbol $\widetilde{p}$ of FDKP defined in \eqref{symbol:pbeta} satisfies $|\widetilde{p}(\xi)| \sim |\xi|^{\frac32}$.} Strichartz estimate for the fractional KP equation with $\alpha=\frac12$ where the loss is slightly smaller than $3/8$ derivatives (see Proposition 4.9 in \cite{LPS1}). 
\end{remark}

As an application of these Strichartz estimates, we are able to improve the standard well-posedness result $H^s(\mathbb R^2)$, $s>2$, for FDKP in the case of capillary-gravity waves ($\beta>0$).
\begin{theorem} \label{WP:theorem}
Assume that $\beta=\epsilon=1$ and $s>\frac74$.
Then, for any $u_0 \in H^s(\mathbb R^2)$, there exist a positive time $T = T(\|u_0\|_{H^s} )$ (which can be
chosen as a nonincreasing function of its argument) and a unique solution $u$ to the
IVP associated to the FDKP equation \eqref{FDKP} in the class
\begin{equation} \label{WP:theorem.1}
C([0,T] : H^s(\mathbb R^2)) \cap L^1((0,T) : W^{1,\infty}(\mathbb R^2) , 
\end{equation}
satisfying $u(\cdot,0)=u_0$.

Moreover, for any $0<T'<T$, there exists a neighborhood $\mathcal{U}$ of $u_0$ in $H^s(\mathbb R^2)$
such that the flow map data-to-solution
\begin{equation*}
 \mathcal{U} \to C([0,T'] : H^s(\mathbb R^2)), \ v_0 \mapsto v \, ,
\end{equation*}
is continuous.
\end{theorem}

We now comment on the main ingredients in the proof of Theorem \ref{WP:theorem}. A standard energy estimate combined with the Kato-Ponce commutator estimate yields 
\begin{equation} \label{EE}
\sup_{t \in [0,T]}\|u(\cdot,t)\|_{H^s_x}^2 \le \|u(0)\|_{H^s}^2+c\left(\int_0^T\|\nabla u(\cdot,t)\|_{L^{\infty}_x}dt\right)\sup_{t \in [0,T]}\|u(\cdot,t)\|_{H^s_x}^2 , \quad s>0 \, .
\end{equation}
Therefore, the main difficulty is to control the term $\|\nabla u\|_{L^1_TL^{\infty}_x}$. This can be done easily using the Sobolev embedding at the \lq\lq hyperbolic\rq\rq \, regularity $s>2$. To lower this threshold, we use a refined Strichartz estimate on the linear non-homogeneous version of \eqref{FDKP} (see Lemma \ref{ref:Str:lemma}). More precisely, after performing a Littlewood-Paley decomposition on the function $u$, we chop the time interval $[0,T]$ into small intervals whose size is inversely proportional to the frequency of the Littlewood-Paley projector. Then, we apply our frequency localised Strichartz estimate (Theorem \ref{lm-LocStr}) to each of these pieces and sum up to get the result. This estimate allows to control   $\|\nabla u\|_{L^1_TL^{\infty}_x}$ at the regularity level $s>\frac74$. Note that similar estimates have already been used for nonlinear dispersive equations (see for instance to \cite{BC,Ta,BGT,KoTz,KeKo,Ke,LPS2,LPS1,HLRRW}).

This estimate combined with the energy estimate \eqref{EE} provides an \emph{a priori} bound on smooth solutions of \eqref{FDKP}. The existence of solutions in Theorem \ref{WP:theorem} is then deduced by using compactness methods, while the uniqueness follows from an energy estimate for the difference of two solutions in $L^2$ combined with Gronwall's inequality. Finally, to prove the persistence property and the continuity of the flow, we use the Bona-Smith argument. 

\medskip    
The paper is organized as follows: in Section 2, we prove the sharp asymptotics for the asymmetric Bessel functions, which will be used to prove Theorem \ref{lm-dispest} and \ref{lm-LocStr} in Section 3. Section 4 is devoted to the proof of the local well-posedness result. Finally, we derive some useful estimates on the derivatives of the symbol $m_{\beta}$ in the appendix.

\medskip
\noindent \textbf{Notation}. 
For any positive numbers $a$ and $b$, the notation $a\lesssim b$ stands for $a\le cb$, where $c$ is a positive constant that may change from line to line. Moreover, we denote $a \sim b$  when  $a \lesssim b$ and $b \lesssim a$.

We also set $\angles{\xi}:= \left(1 +|\xi|^2\right)^\frac12.$

For $x=(x_1,x_2)$, $u=u(x)\in \mathcal{S}'(\R^2)$, $\mathcal{F} u=\hat{u}$ will denote the Fourier transform of $u$. For $s\in\R$, we define the Bessel potential of order $-s$, $J^s$ by
$$
J^su = \mathcal{F}^{-1}(\angles{\xi}^s\mathcal{F}u) \, .
$$

Throughout the paper, we fix a smooth cutoff function $\chi$ such that
\begin{equation}\label{eta}
\chi \in C_0^{\infty}(\mathbb R), \quad 0 \le \chi \le 1, \quad
\chi_{|_{[-1,1]}}=1 \quad \mbox{and} \quad  \mbox{supp}(\chi)
\subset [-2,2].
\end{equation}
 We set
$$
\rho(s)
=\chi\left(s\right)-\chi \left(2s\right).
 $$
 Thus, $\supp \rho= 
\{ s\in \R: 1/ 2 \le |s| \le 2 \}$. For 
  $\Lambda \in  2^\Z$ we set $\rho_{\Lambda}(s):=\rho\left(s/\Lambda\right)$
 and define the frequency projection $P_\Lambda$ by
\begin{align*}
\widehat{P_{\Lambda} f}(\xi)  = \rho_\Lambda(|\xi|)\widehat { f}(\xi) .
 \end{align*}
 Any summations over capitalized variables such as $\Lambda$ or $\Gamma$ are presumed to be over dyadic numbers. We also define
 \[ P_{\le \Lambda} =\sum_{\Gamma \le \Lambda} P_{\Gamma} \quad \text{and} \quad P_{>\Lambda}=1-P_{\le \Lambda}=\sum_{\Gamma>\Lambda}P_{\Gamma} \, .\]
We sometimes write $f_\Lambda:=P_\Lambda f $, so that
\[ f=\sum_{\Lambda \in 2^\Z } f_\Lambda = P_{\le 1}f+\sum_{\Lambda > 1}f_{\Lambda} .\]
 
For $1\le p\le\infty$, $L^p(\mathbb R^2)$ denotes the usual Lebesgue space and for $s\in\R$, $H^s(\mathbb R^2)$ is the $L^2$-based Sobolev space with norm $\|f\|_{H^s}=\|J^s f\|_{L^2}$.
If $B$ is a space of functions on $\R^2$, $T>0$ and $1\le p\le\infty$, we define the spaces $L^p\big((0,T) : B \big)$ and $L^p\big( \mathbb R : B\big)$ respectively through the norms
$$
\|f\|_{L^p_TB_x} = \left( \int_0^T \|f(\cdot,t)\|_{B}^p dt \right)^{\frac1p} \quad  \textrm{and} \quad \|f\|_{L^p_tB_x} = \left( \int_{\mathbb R} \|f(\cdot,t)\|_{B}^p dt \right)^{\frac1p} \, ,
$$
when $1 \le p < \infty$, with the usual modifications when $p=+\infty$.

\section{ Identities and decay for the asymmetric Bessel functions}

\begin{proposition}[Identities and decay for $J_+$] \label{propJ+}
Let $x=(x_1, x_2)$, $s_1=\sgn(x_1)$ and  $x'_2=x_2 /|x|$. Define
\begin{align*}
J_+( x) &= \int_{-\pi/2}^{\pi/2}  e^{i x \cdot  \omega(\theta) }  \, d\theta,
\end{align*}
where $\omega (\theta)=(\cos \theta, \sin  \theta)$.
Then we have the following:
\begin{enumerate}
[(i).] 
\item \label{prop-J+ident}
$J_+$ can be written as
\begin{equation}\label{J+ident}
J_+( x)=
 F(| x|, | x'_2|) + F^{s_1}(|x|, | x'_2| ),
\end{equation}
where
 \begin{align*}
F(r, a )
&= \int_{-a}^{a}  e^{ir  s }  \frac{ds}{\sqrt{1-s^2}} , \\
F^{\pm} (r, a )
&= 2\int_{a}^{1}  e^{\pm i  r  s}  \frac{ds}{\sqrt{1-s^2}} 
\end{align*}
for $a\in [0,1]$.
\item  \label{prop-F-ident}The functions
  $F$ and $  F^\pm$ can be written as
\begin{align}\label{F-ident}
F(r,  a ) & = e^{  ia r} f^+_a (r) +  e^{ - ia  r}  f^-_a (r)  ,
\\
\label{Fpm-ident}
  F^{\pm}  (r, a ) & =   2 e^{ \pm i  r} f^{\pm} _1(r) - 2e^{  \pm  i a   r} f^{\pm}  _a(r),
\end{align}
where
\begin{equation}
\label{fa}
\begin{split}
f^\pm _{a}(r)&=  \mp i \int_{0}^{\infty}  e^{-rs} \left(s^2+1-a^2 \mp  2ais \right)^{-1/2} \, ds.
\end{split}
\end{equation}

\item  \label{prop-fa-decay}  Moreover, the functions $ f^\pm_a $ and their derivatives satisfy the decay estimates
\begin{equation}\label{fa-est}
 \Bigabs{\partial_r^j f^\pm_a(r)} \le C r^{-j-1/2} \qquad (j=0, 1) 
\end{equation}
for all $r\ge 1$ and $a\in [0,1]$.

\end{enumerate}

\end{proposition}

The proof of Proposition \ref{propJ+} is given in the following subsections.

\subsection{Proof of Proposition \ref{propJ+} \eqref{prop-J+ident} }
Writing $x=|x|\omega (\alpha)$, where $\alpha  \in [0, 2\pi)$, we have
\begin{align*}
J_+(x) & = \int_{-\pi/2}^{\pi/2}  e^{i |x|  \omega(\alpha)  \cdot \omega(\theta) }  \, d\theta
= \int_{-\pi/2}^{\pi/2}  e^{i |x|  \cos (\theta-\alpha) }  \, d\theta
\\
&= \int_{\alpha-\pi/2}^{\alpha+\pi /2}  e^{i|x|  \cos (\theta) }  \, d\theta.
\end{align*}
We shall use the following change of variables:
$$s=\cos \theta    \ \Rightarrow \sin \theta = \pm \sqrt{1-s^2} ,  \ ds= -\sin \theta  d\theta.
$$
Now if $\alpha \in [0, \pi/2]$, i.e., $x_1\ge 0$ and $x_2\ge 0$, we write 
\begin{align*}
J_+( x) & =  \left( \int_{\alpha-\pi/2}^{0}  + \int_{0}^{\alpha+\pi/2}  \right)   e^{i|x|  \cos (\theta) }  \, d\theta
\\
&= \left(\int_{\sin\alpha}^{1} + \int_{-\sin\alpha}^{1} \right)  e^{i  |x|   s}  \frac{ds}{\sqrt{1-s^2}} 
\\
&=\left(\int_{-\sin\alpha}^{\sin\alpha} + 2 \int_{\sin\alpha}^{1} \right)  e^{i  |x|   s }  \frac{ds}{\sqrt{1-s^2}} 
\\
&= F(|x|,  |x'_2|) +  F^+ (|x|, |x'_2| ) ,
\end{align*}
where we used the fact that $x_2'=x_2/|x|=\sin \alpha >0$ and $s_1=\sgn(x_1)=+$.

If $\alpha \in [ \pi/2, \pi)$, i.e., $x_1\le 0$ and $x_2> 0$,  we split the integral over $[\alpha-\pi/2, \pi]$  and  $[\pi, \alpha+\pi/2]$, and write
\begin{align*}
J_+( x) 
&= \left(  \int_{-\sin \alpha}^{\sin \alpha} + 2 \int_{-1}^{-\sin \alpha} \right)  e^{i    |x| s}  \frac{ds}{\sqrt{1-s^2}}\
\\
&= F(|x|, | x'_2|) + F^-(|x|, |x'_2| )  .
\end{align*}
The remaining cases can be established similarly. In fact,  if $\alpha \in [ \pi, 3\pi/2)$, i.e., $x_1< 0$ and $x_2\le 0$, we split the integral over $[\alpha-\pi/2, \pi]$ and  $[\pi, \alpha+\pi/2]$ whereas if $\alpha \in [  3\pi/2, 2\pi)$, i.e., $x_1\ge 0$ and $x_2< 0$,  we split the integral over $[\alpha-\pi/2, 2\pi]$  and  $[2\pi, \alpha+\pi/2]$ to obtain the desired identities.

\subsection{Proof of Proposition \ref{propJ+} \eqref{prop-F-ident} }We follow \cite[Chapter 4, Lemma 3.11]{SW71}.
For fixed $0<\delta \ll 1$ and $R\gg 1$, let $\Omega_{\delta}(a,  R)$ be the region in the complex plane obtained from the rectangle with vertices at points
$(-a,0)$,  $(a,0)$, $(a, R)$ and $(-a, R)$, by removing two quarter circles 
 of radius $\delta$ and centered at $(a, 0)$ and $(-a, 0)$, denoted $C_\delta(a)$ and $C_\delta(-a)$, respectively; see fig. 1 below.

The functions
 $$
 h^\pm(z)=e^{\pm irz}(1-z^2)^{-1/2}
 $$ 
 have no poles in 
$\Omega_{\delta}.$
So by Cauchy's theorem we have
\begin{align*}
0 &= \int_{ \partial \Omega_{\delta} }h^+(z) \, dz 
\\
&= \int_{-a+\delta}^{a-\delta}  h^+(s) \, ds
 +  i \int_{\delta}^{R}  h^+(a+is)  \, ds -  i \int_{\delta}^{R}  h^+(-a+is)  \, ds 
+\mathcal E_\delta(a, R),
\end{align*}
 where
 $$\mathcal E_\delta(a, R) =   \int_{C_\delta(a)}  h^+ (z)\, dz +   \int_{C_\delta(-a)}  h^+ (z)\, dz -  \int_{-a}^{a}  h^+ (s+iR)\, ds  .
 $$
 Now
 letting $\delta \rightarrow 0$ and  $R \rightarrow \infty$, one can show that $\mathcal E(\delta, R)  \rightarrow 0$, and hence 
\begin{align*}
F(r, a)= \int_{-a}^{a}  e^{irs}(1-s^2)^{-1/2} \, ds
 &= i e^{-air}  \int_{0}^{\infty}  e^{-rs} \left(s^2+1-a^2+ 2ais \right)^{-1/2} \, ds
 \\
 \qquad &
-i  e^{iar}   \int_{0}^{\infty}  e^{-rs} \left(s^2+1-a^2 - 2ais \right)^{-1/2} \, ds
\\
&=e^{iar} f^+ _a(r)+   e^{-iar} f^- _a(r)
\end{align*}
which proves \eqref{F-ident}.

The identity \eqref{Fpm-ident} is proved in a similar way.
Indeed,
let $\Omega^\pm_{\delta}(a,  R)$  
be the region in the complex plane obtained from the rectangle with vertices at points $(a,0)$,  $(1,0)$, $(1, \pm R)$ and $(a, \pm R)$, by removing the quarter circle 
 of radius $\delta$ and centered at $(1, 0)$, denoted $C_\delta$; see fig. 2 below.

Again, the functions $h^\pm( z)$ have no poles in 
$
\Omega^\pm_{\delta}.
$
So by Cauchy's theorem we have
\begin{align*}
0 &= \int_{  \partial \Omega^+ _{\delta} } h^+ (z)\, dz 
\\
&=  \int_{a}^{1-\delta} h^+(s) \, ds
 +  i \int_{\delta}^{R}  h^+ (1+is)\, ds 
- i \int_{0}^{R}  h^+(a+is) \, ds +\mathcal E_\delta^+(a, R),
\end{align*}
 where 
 $$\mathcal E_\delta^+(a, R) =   \int_{C_\delta}  h^+ (z)\, dz -  \int_{a}^{1}  h^+ (s+iR)\, ds  .$$
 Letting $\delta \rightarrow 0$ and  $R \rightarrow \infty$, one can show that $ \mathcal E_\delta^+(a, R)  \rightarrow 0$, and hence 
\begin{align*}
 F^+(r, a)&=2 \int_{a}^{1}  e^{irs}(1-s^2)^{-1/2} \, ds
\\
 &= 2i e^{air}  \int_{0}^{\infty}  e^{-rs} \left(s^2+1-a^2- 2ais \right)^{-1/2} \, ds 
  \\ 
  & \qquad   \qquad  -2i  e^{ir}   \int_{0}^{\infty}  e^{-rs} \left(s^2- 2is \right)^{-1/2} \, ds
\\
&=  2e^{ ir} f^+_1(r) -   2e^{ air} f^+_ a (r)  .
\end{align*}
Similarly, integrating $h^-(z)$ over 
$
\partial \Omega^-_{\delta},
$
one can show
\begin{align*}
 F^-(r, a)
&=   2e^{- ir} f^-_1(r) -2e^{ -air} f^-_ a (r)    .
\end{align*}

\begin{figure}\label{fig1}
\begin{center}
\begin{tikzpicture}[scale=1]
\draw[thick, ->] (-4,0) -- (4,0) node[right, dashed] {$\re z$};
\draw[thick, ->] (0,5) -- (0,6) node[right, dashed] {$\im z$};

\draw [thick, -> ]  (-2, 0)--(0,0);
\draw [thick, -> ]  (3, 2)--(3, 2.5);
 \draw [thick, -> ]  (2, 5)--(1, 5);
 \draw [thick, -> ]  (-3, 2.5)--(-3, 2);
\draw [thick, -> ]  (0, 5)--(-.5, 5);

\node at (0,2.5){$\Omega_\delta$};

\node at (0,5){$\bullet$};

\node at (-3,0){$\bullet$};
\node at (-3, -.25){$-a$};
\node at (3,0){$\bullet$};
\node at (3, -.25){$a$};
\node at (0, -.25){$0$};
  \coordinate(A) at (2, 0);
  \draw (A) arc[start angle=180, end angle=90, radius=1];
\node at (2,-.25){$a-\delta$};
\node at (2,0){$\bullet$};
  \coordinate(B) at (-2, 0);
  \draw (B) arc[start angle=0, end angle=90, radius=1];
\node at (-2,-.25){$-a+\delta$};
\node at (-2,0){$\bullet$};
\draw[ color=black] (-3,1) -- (-3,5);
\draw[ color=black] (3,1) -- (3,5);
\draw[ color=black] (-3,5) -- (3,5);
\node at (.25, 5.25){$R$};
\end{tikzpicture}
\end{center}
\caption{The region $\Omega_\delta(a,R)$ with $0\le a\le 1$, $0<\delta\ll1 $ and $R\gg1 $.}
\end{figure}
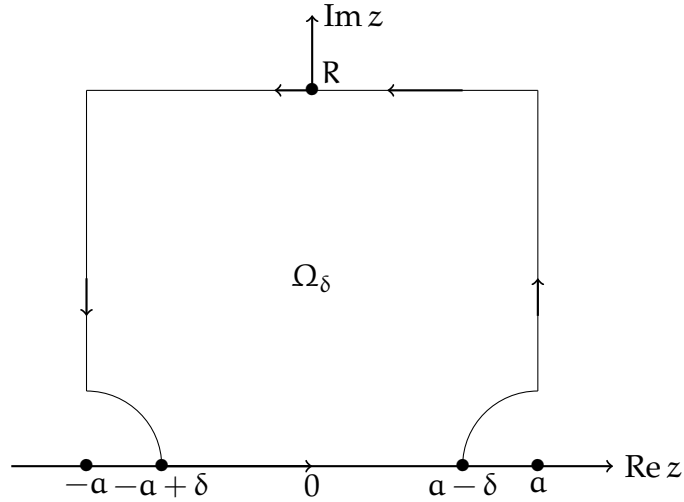

\begin{figure}\label{fig2}
\begin{center}
\begin{tikzpicture}[scale=1]
\draw[->] (0,0) -- (6,0) node[right, dashed] {$\re z$};
\draw[->] (0,-3.25) -- (0,4) node[right, dashed] {$\im z$};
\node at (2.5,-.2){$1-\delta$};

 \draw [thick, -> ]  (1.5, 0)--(1.7, 0);
 \draw [thick, -> ]  (4, 2)--(4, 2.2);
 \draw [thick, -> ]  (3, 3)--(2.5, 3);
 \draw [thick, -> ]  (1, 2)--(1, 1.5);

 \draw [thick, -> ]  (4, -1)--(4, -1.5);
 \draw [thick, -> ]  (3, -3)--(2.5,-3);
 \draw [thick, -> ]  (1, -1.5)--(1, -1);

 \draw [thick, -> ]  (3.05, .3)--(3.1, .5);
 \draw [thick, -> ]  (3.05, -.3)--(3.1, -.5);

\node at (3,0){$\bullet$};
\node at (-.25, -.25){$0$};
\node at (4,0){$\bullet$};
\node at (4,-.25){$1$};
\draw[thick, color=black] (1,-3) -- (1,3);
\draw[thick, color=black] (4,1) -- (4, 3);
\draw[thick, color=black] (4,-1) -- (4, -3);
\draw[thick, color=black] (1,-3) -- (4,-3);
\draw[thick, color=black] (1,3) -- (4,3);

  \coordinate(A) at (3, 0);
  \draw (A) arc[start angle=180, end angle=90, radius=1];
    \draw (A) arc[start angle=180, end angle=270, radius=1];

\node at (3,1.5){$\Omega^+_\delta$};
\node at (3,-1.5){$\Omega^-_\delta$};

\node at (.75,-.25){$a$};
\node at (1,0){$\bullet$};
\node at (-.3, 3){$R$};
\node at (0,3){$\bullet$};
\node at (0,-3){$\bullet$};

\node at (-.37, -3){$-R$};

\end{tikzpicture}
\end{center}
\caption{The regions $\Omega^+_\delta(a,R)$ and $\Omega^-_\delta(a,R)$ with $0\le a\le 1$, $0<\delta\ll1 $ and $R\gg1 $.}
\end{figure}
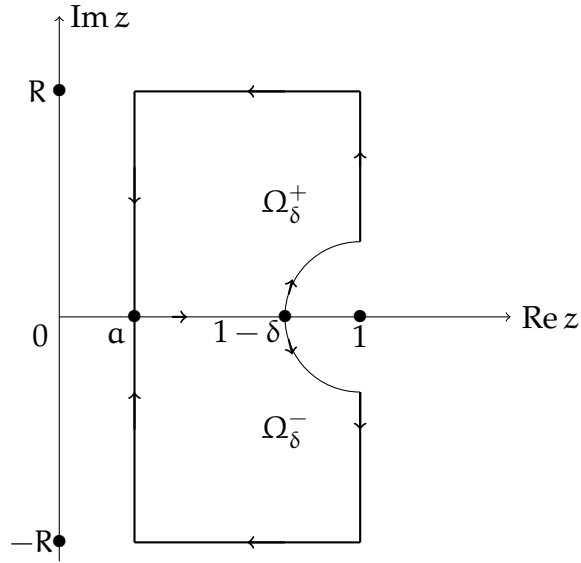

\subsection{Proof of Proposition \ref{propJ+} \eqref{prop-fa-decay} }
Observe that the following estimate holds for all $s\ge 0$ and $0\le a\le 1$:
\begin{equation*}
\Bigabs{\left( s^2-a^2+1 \pm 2ais \right)^{-\frac12} } \le \left ( \max( s^2+1-a^2 ,\  2as)\right)^{-\frac12}.
\end{equation*}
We treat the cases $0 \le a \le 1/\sqrt{2}$ and $ 1/\sqrt{2} \le a \le 1$ separately.
 
\textbf{Case 1: $0 \le a \le 1/\sqrt{2}$.} In this case 
we have
$
 2as\le   s^2+1-a^2$, and hence
\begin{align*}
|f^\pm _a (r)| &\le    \int_{0}^{\infty}  e^{-rs} \left ( s^2+1-a^2\right)^{-\frac12} \, ds
\\
& \lesssim \frac 1{\sqrt{1-a^2} } \int_{0}^{\sqrt{1-a^2} }   e^{-rs} \, ds +   \int_{\sqrt{1-a^2}}^{\infty} e^{-rs} s^{-1} \, ds
\end{align*}
Now using the fact that $1/2\le 1-a^2 \le 1$ and $r \ge 1$, we bound the first integral on the right by
$$\frac{ 1 }{\sqrt{2} }  \int_{0}^{1 }   e^{-rs} \, ds= \frac{ 1-e^{-r} }{\sqrt{2} r}
\sim r^{-1}.$$  Similarly, the second integral on the right is bounded by 
$$
 \frac 1{\sqrt{2} } \int_{1/\sqrt2}^{\infty }   e^{-rs} \, ds \lesssim  r^{-1 }  e^{-r/\sqrt{2}}.
$$
Thus, $|f^\pm_a (r)| \lesssim  r^{-1 }$.

For the derivative, we simply estimate
\begin{align*}
|\partial_r f^\pm_a (r)| &\le    \int_{0}^{\infty}  e^{-rs} s \left ( s^2+1-a^2\right)^{-\frac12} \, ds
\\
& \lesssim \frac1{\sqrt{1-a^2} } \int_{0}^{\sqrt{1-a^2} }   e^{-rs} s \, ds +   \int_{\sqrt{1-a^2}}^{\infty} e^{-rs}   \, ds 
\\
& \lesssim r^{-2} +  r^{-1} \int_{r/\sqrt{2} }^{\infty} e^{-s}  \, ds 
\\
& \lesssim r^{-2} +   r^{-1 }  e^{-r/\sqrt{2}}  
\lesssim r^{-2} .
\end{align*}

\textbf{Case 2: $  1/\sqrt{2} \le a \le 1$.
} We split 
\begin{align*}
f^\pm_a (r)&=   i\left(\int_{0}^{a}  +   \int_{a}^{\infty}\right) e^{-rs} \left(s^2+1-a^2\mp  2ais \right)^{-1/2} \, ds.
\end{align*}
Then the first integral on the right-hand side is bounded by 
\begin{align*}
  \int_{0}^{a}  e^{-rs} (2as)^{-\frac12} \, ds& =\frac{1}{\sqrt{2ar}} \int_{0}^{ar }  e^{-s} s^{-\frac 12} \, ds
\\
& \le \frac{1}{\sqrt{2ar}} \left( \int_{0}^{a } s^{-\frac 12} \, ds + a^{-\frac12} \int_{a}^{ar}   e^{-s} \, ds \right) \lesssim r^{-\frac12}
\end{align*}
and the second integral on the right is bounded by 
\begin{align*}
  \int_{a}^{\infty}  e^{-rs} \left ( s^2+1-a^2\right)^{-\frac12} \, ds & \le \int_{a}^{\infty}  e^{-rs} s^{-1} \, ds
\\
& =   \int_{ar}^{\infty}  e^{-s}  s^{-1} \, ds \le (ar)^{-1 }  e^{-ar} \lesssim r^{-1}.
\end{align*}
Thus, $| f^\pm_a (r) | \lesssim r^{-\frac12} .$

For the derivative, we simply estimate
\begin{align*}
| \partial_r f^\pm_a (r)  | \le   \int_{0}^{\infty}  e^{-rs} s (2as)^{-\frac12} \, ds& =\frac{r^{-\frac32}}{\sqrt{2a}} \int_{0}^{\infty }  e^{-s} s^{\frac 12} \, ds
&
\\
\lesssim r^{-\frac32} \int_{0}^{\infty }  e^{-s/2}  \, ds \lesssim  r^{-\frac32}.
\end{align*}

\section{localised dispersive and Strichartz estimates}
 
Van der Corput's Lemma will be useful in the proof of Theorem \ref{lm-dispest}.
\begin{lemma}[Van der Corput's Lemma, \cite{S93} ]\label{lm-corput}
 Assume 
$g \in C^1(a, b)$, $\psi\in  C^2(a, b)$ and $|\psi''(r)|  \geq A$ for all $r\in (a, b)$. Then 
\begin{align}
\label{corput} 
\Bigabs{\int_a^b e^{i t  \psi(r)}   g(r) \, dr}& \le C  (At)^{-1/2}  \left[ |g(b)| + \int_a^b |g'(r)| \, dr \right] ,
\end{align}
for some constant $C>0$ that is independent of $a$, $b$ and $t$.
\end{lemma}

\begin{proof}[Proof of Theorem \ref{lm-dispest}] We may assume $t>0$. Now 
 recalling the definition of $S_{m_\beta}$ in \eqref{lin:KPFD}--\eqref{def:m_beta}, we can write
 \begin{align*}
\left[ S_{m_{\beta}} (t) f_\Lambda \right](x) & = ( I_{\Lambda, t} \ast f)(x),
\end{align*}
where $ I_{\Lambda, t}$ is as in \eqref{I-lamb}-\eqref{asymbesselfunc}, i.e., 
$$
  I_{\Lambda, t} (x)=\Lambda^2  \int_0^\infty\left[ e^{it m_\beta( \Lambda r)}J_+(\Lambda rx) +  e^{-itm_\beta( \Lambda r)}J_-( \Lambda r x)  \right]\tilde \rho(r)   \, dr \, ,
 $$
with $\tilde \rho(r) =r \rho(r) $ and
\begin{equation*}
J_\pm(x) = \int_{\omega \in S^1} \mathbb {1}_{\{\pm \omega_1>0\}}  e^{i   x \cdot \omega}  \, d\sigma(\omega) .
\end{equation*}

By Young's inequality
\begin{equation}\label{younginq}
\| S_{m_{\beta}} (t) f_\Lambda  \|_{L^\infty_x(\R^2)} \le \| I_{\Lambda, t} \|_{L^\infty_x(\R^2)}  \|f\|_{L_x^1(\R^2)}
\end{equation}
Thus, \eqref{dispest} reduces to
\begin{align} 
\label{Idispest}
\norm{I_{\Lambda, t} }_{L^\infty_x(\R^2)} &\lesssim   \angles{ \sqrt{  \beta} \Lambda }^{-1} \angles{ \Lambda }^{\frac32}t^{-1 }   .
\end{align}

Using the parametrization $ \omega(\theta)=(\cos \theta, \sin  \theta)$, we can write 
\begin{align*}
J_+( x) &= \int_{-\pi/2}^{\pi/2}  e^{i x \cdot  \omega(\theta) }  \, d\theta
\end{align*}
and 
\begin{align*}
J_-( x) &= \int_{\pi/2}^{3\pi/2}  e^{i x \cdot  \omega(\theta) }  \, d\theta=\int_{-\pi/2}^{\pi/2}  e^{-i x \cdot  \omega(\theta)}  \, d\theta= \overline{ J_+(x)}.
\end{align*}
Thus,
\begin{equation*}
I_{\Lambda, t} (x) = 2 \Lambda^{2}   \re  \widetilde I_{\Lambda, t} (x) , 
\end{equation*}
where
\begin{equation}
\label{Itildedef1}
 \widetilde I_{\Lambda, t} (x) 
= \int_{1/2}^2  e^{itm_\beta (\Lambda r)}J_+(\Lambda r x)  \tilde \rho(r)   \, dr.
\end{equation}
So \eqref{Idispest} reduces to proving
\begin{align} 
\label{Itdispest}
\norm{ \widetilde I_{\Lambda, t} }_{L^\infty_x(\R^2)} &\lesssim  \Lambda^{-2} \angles{ \sqrt{  \beta} \Lambda }^{-1} \angles{ \Lambda }^{\frac32}  t^{-1 }  .
\end{align}

We treat the cases $  |x|\lesssim \Lambda^{-1}$ and $  |x|\gg   \Lambda^{-1}$ separately. 
First assume $  |x|\lesssim   \Lambda^{-1}$. 
Then for all $ r\in (1/2, 2)$  and $k =0, 1$, we have
\begin{equation}\label{Jest}
 \Bigabs{\partial_r^k J_+(\Lambda r x) } \le  ( \Lambda |x|)^k \lesssim 1.
\end{equation}
Integration by parts yields
\begin{align*}
\widetilde  I_{\Lambda, t}(x)
&= -i (\Lambda t)^{-1}  \int_{1/2}^2  \frac{d}{dr}\left\{ e^{it m_\beta (\Lambda r) }\right\}  [ m_\beta'(\Lambda r) ]^{-1} J_+( \Lambda r x) \tilde\rho(r) \, dr
\\
&=i (\Lambda t)^{-1}  \int_{1/2}^2   e^{it m_\beta (\Lambda r) }[m_\beta'(\Lambda r) ]^{-1} \partial_r\left[  J( \Lambda r  x) \tilde\rho(r) \right] \, dr
\\
 & \qquad -i(\Lambda t)^{-1}   \int_{1/2}^2   e^{it m_\beta (\Lambda r) } [m_\beta'(\Lambda r) ]^{-2} \Lambda m_\beta''(\Lambda r)  J( \Lambda r x) \tilde\rho(r)  \, dr.
\end{align*}
Now applying Lemma \ref{lm-mest} 
and \eqref{Jest} 
we obtain 
\begin{align*}
|  \widetilde I_{\Lambda, t}(x)| &\lesssim (\Lambda t)^{-1} \angles{\sqrt \beta \Lambda}^{-1} \left( \angles{ \Lambda}^{1/2} + \Lambda^2 \angles{\Lambda}^{-\frac 32}\right)
\\
&\lesssim \Lambda^{-1} \angles{\sqrt \beta \Lambda}^{-1}  \angles{\Lambda}^{\frac 12} t^{-1} 
\\
&\lesssim  \Lambda^{-2}\angles{\sqrt \beta \Lambda}^{-1}  \angles{\Lambda}^{\frac 32} t^{-1} .
\end{align*}
 
 So from now on we assume
$  |x|\gg   \Lambda^{-1}$. Let $x=(x_1, x_2)$, $s_1=\sgn(x_1)$ and
 $x'_2=x_2 /|x|$.
By Proposition \ref{propJ+}\eqref{prop-J+ident}--\eqref{prop-F-ident} we can write 
\begin{equation}
\label{Jiden}
\begin{split}
J_+(\Lambda r x) 
&=2 e^{ i s_1 \Lambda  |x| r} f^{s_1}_1( \Lambda |x| r) - s_1 e^{    i   \Lambda |x_2|  r} f^+ _{| x_2'|}( \Lambda r|x|)
\\
& \qquad +   s_1  e^{ - i \Lambda | x_2| r}  f^-_{| x_2' |} (\Lambda |x| r),
\end{split}
\end{equation}
whereas by Proposition \ref{propJ+}\eqref{prop-fa-decay} the functions $ f^\pm_{a} $, with $a=1$ or $| x_2'|$,
satisfy the estimates
\begin{equation}\label{falm-est}
 \Bigabs{\partial_r^j \left[ f^\pm_a  ( \Lambda r|x|) \right]} \lesssim  ( \Lambda |x|)^{-\frac 12} \qquad (j=0, 1) 
\end{equation}
for all $r\in [1/2, 2]$.
Set 
\begin{equation}
\label{Hiden}
\begin{split}
G^{\pm}_\Lambda (x, r):&= f^{\pm}_1 ( \Lambda r|x|) \tilde\rho(r) , \\
H_\Lambda^\pm (x, r):&= f^\pm_{|x_2'|}  ( \Lambda r|x|) \tilde\rho(r) .
\end{split}
\end{equation}
Then by \eqref{falm-est} we have 
\begin{equation}
\label{Hest}
| \partial_r^j G^{\pm}_\Lambda (x, r) | + | \partial_r^j H_\Lambda^\pm (x, r) | \lesssim   (\Lambda |x|)^{-\frac 12}  \qquad (j=0, 1). 
\end{equation}

Now using \eqref{Jiden} and \eqref{Hiden} in \eqref{Itildedef1} we can write
\begin{align}
\label{Itilsplit}
\widetilde I_{\Lambda, t} (x) =   2 
 \mathcal I^{s_1}_{\Lambda, t} (x) - s_1 \mathcal J^+_{\Lambda, t} (x) + s_1 \mathcal J^-_{\Lambda, t} (x) ,
\end{align}
where 
\begin{align*}
\mathcal I^{\pm}_{\Lambda, t} (x) &=   \int_{1/2}^2  e^{it\phi^{\pm}_{\Lambda} (x, r)} G^{\pm}_\Lambda(x, r) \, dr
 \\
\mathcal  J^\pm_{\Lambda, t} (x) &=  \int_{1/2}^2  e^{it\psi^\pm_{\Lambda} (x, r)} H_\Lambda^\pm (x, r) \, dr
\end{align*}
with 
\begin{align*}
\phi^{\pm}_{\Lambda} (x, r)&=   m_\beta (\Lambda r)  \pm \Lambda   |x| r/t  ,
\\
\psi^\pm_{\Lambda} (x, r) &=   m_\beta (\Lambda r) \pm   \Lambda |x_2| r/t  .
\end{align*}

Observe that
\begin{align*}
\partial_r \phi^{\pm}_{\Lambda} (x, r) =  \Lambda  \left[   m_\beta'(\Lambda r) \pm  |x|/t  \right],\quad \partial_r^2 \phi^{\pm}_{\Lambda} (x, r) =      \Lambda^2 m_\beta''(\Lambda r),
\\
\partial_r \psi^\pm_{\Lambda} (x, r) =  \Lambda  \left[   m_\beta'(\Lambda r) \pm |x_2|/  t  \right],\quad \partial_r^2 \psi^{\pm}_{\Lambda} (x, r) =      \Lambda^2 m_\beta''(\Lambda r),
\end{align*}
By Lemma \ref{lm-mest}, we have
\begin{equation}
\label{ph''est}
|\partial^2_r  \phi^{\pm}_{\Lambda} (x, r)  |  \sim  |\partial^2_r  \psi^\pm_{\Lambda} (x, r)  | \sim \Lambda^3 \angles{\sqrt \beta \Lambda} \angles{ \Lambda }^{-\frac 52} 
\end{equation}
for all $ r\in (1/2, 2)$.

First we estimate $ \mathcal J^+_{\Lambda, t} (x) $ and $ \mathcal J^-_{\Lambda, t} (x) $. The same argument works for  $ \mathcal I^{\pm}_{\Lambda, t} (x)$, and we shall comment on this below.

\smallskip
\noindent \textit{Estimate for $ \mathcal J^+_{\Lambda, t} (x)$.} By
  Lemma \ref{lm-mest}, we have
\begin{equation}
\label{ph'+est1}
|\partial_r  \psi^+_{\Lambda} (x, r)  | \gtrsim   \Lambda \angles{\sqrt \beta \Lambda}  \angles{ \Lambda }^{-\frac 12} 
\end{equation}
where to obtain this lower bound we also used the fact that $m_\beta'$ is positive.

Integration by parts yields
\begin{align*}
  \mathcal J^+_{\Lambda, t} (x)
&=-  i   t^{-1} \int_{1/2}^2 \partial_r\left[  e^{it \psi^+_{\Lambda} (x, r)  } \right]  \left[\partial_r  \psi^+_{\Lambda} (x, r) \right]^{-1} H_\Lambda^+ (x, r)  \, dr
\\
&=i  t^{-1}  \int_{1/2}^2  e^{it \psi^+_{\Lambda} (x, r)   }  \left\{
\frac{\partial_r H_\Lambda^+ (x, r) }{\partial_r  \psi^+_{\Lambda} (x, r)  } -   \frac{\partial^2_r  \psi^+_{\Lambda} (x, r)  H_\Lambda^+ (x, r)  }{\left[\partial_r \psi^+_{\Lambda} (x, r)  \right]^{2}}\right\}\, dr.
\end{align*}
Then using \eqref{Hest},   \eqref{ph''est}  and  \eqref{ph'+est1} we obtain
\begin{equation*}
\begin{split}
|  \mathcal J^+_{\Lambda, t} (x)|
&\lesssim   t^{-1} \angles{\sqrt \beta \Lambda}^{-1}  \left\{ \Lambda^{-1}  \angles{ \Lambda }^{\frac 12}  +    \Lambda  \angles{\Lambda }^{-\frac 32}   \right\}(\Lambda |x|)^{-\frac 12}
\\
&\lesssim  \Lambda^{-1} \angles{\sqrt \beta \Lambda}^{-1} \angles{  \Lambda }^{\frac 12}  t^{-1} 
\\
&\lesssim    \Lambda^{-2}\angles{\sqrt \beta \Lambda}^{-1}  \angles{\Lambda}^{\frac32} t^{-1} .
\end{split}
\end{equation*}

\noindent \textit{Estimate for $ \mathcal J^-_{\Lambda, t} (x)$.}
We treat the non-stationary case
\begin{equation}
\label{stcase}
 | x_2| \ll   \angles{\sqrt \beta \Lambda} \angles{  \Lambda }^{-\frac 12} t \quad \text{or} \quad |x_2| \gg  \angles{\sqrt \beta \Lambda} \angles{  \Lambda }^{-\frac 12}t
\end{equation}
and the stationary case
\begin{equation}
\label{x2stcase}
 | x_2| \sim  \angles{\sqrt \beta \Lambda} \angles{  \Lambda }^{-\frac 12} t
\end{equation}
separately. 

In the non-stationary
case \eqref{stcase}, we have 
$$
|\partial_r \psi^-_{\Lambda} (x, r)|\gtrsim \Lambda  \angles{\sqrt \beta \Lambda} \angles{  \Lambda }^{-\frac 12}.
$$
Hence $ \mathcal J^-_{\Lambda, t}(x)$ can be estimated in exactly the same way as $ \mathcal J^+_{\Lambda,   t}(x)$ , and satisfy the same bound.

So it remains to treat the stationary case. In this case,  
we use Lemma \ref{lm-corput}, \eqref{ph''est}, \eqref{Hest} and \eqref{x2stcase} to obtain 
\begin{equation}\label{Iest-station}
\begin{split}
| \mathcal J^-_{\Lambda, t} (x)| &= \Bigabs{\int_{1/2}^2  e^{it\psi^-_{\Lambda} (x, r)} H_\Lambda^- (x, r) \, dr}
\\
&\lesssim  t^{-\frac12}   \Lambda^{-\frac32}  \angles{\sqrt \beta \Lambda}^{-\frac12} \angles{  \Lambda }^{\frac54 }  \left[ |H_\Lambda^- (x, 2) |+ \int_{1/2}^2 | \partial_r  H_\Lambda^- (x, r)| \, dr\right]
\\
&\lesssim  t^{-\frac12}   \Lambda^{-\frac32}  \angles{\sqrt \beta \Lambda}^{-\frac12} \angles{  \Lambda }^{\frac54 } \cdot (\Lambda |x|)^{-\frac 12}
\\
& \lesssim    \Lambda^{-2}   \angles{\sqrt \beta \Lambda}^{-1}\angles{ \Lambda }^{\frac32 } t^{-1}  ,
\end{split}
\end{equation}
where we also used the fact that $H_\Lambda^- (x, 2) =0$,  and  \eqref{x2stcase} which also implies
$$ |x| \ge |x_2|\sim\angles{\sqrt \beta \Lambda} \angles{  \Lambda }^{-\frac 12} t. $$


\noindent \textit{Estimate for $ \mathcal I^{+}_{\Lambda, t} (x)$.}
By
  Lemma \ref{lm-mest} we have 
\begin{equation*}
|\partial_r  \phi^{+}_{\Lambda} (x, r)  | \gtrsim    \ \Lambda  \angles{\sqrt \beta \Lambda} \angles{  \Lambda }^{-\frac 12}.
\end{equation*}
Then integrating by parts and  using the estimates
\eqref{Hest} and \eqref{ph''est}  we obtain
\begin{equation*}
| \mathcal I^{+}_{\Lambda, t}(x)|
\lesssim   \Lambda^{-2}   \angles{\sqrt \beta \Lambda}^{-1}\angles{ \Lambda }^{\frac32 } t^{-1} .
\end{equation*}

\noindent \textit{Estimate for $ \mathcal I^{-}_{\Lambda, t} (x)$.}
 In the non-stationary
case
\begin{equation*}
 | x| \ll   \angles{\sqrt \beta \Lambda} \angles{  \Lambda }^{-\frac 12} t \quad \text{or} \quad |x| \gg  \angles{\sqrt \beta \Lambda} \angles{  \Lambda }^{-\frac 12}t
\end{equation*}
 we have 
$$
|\partial_r  \phi^{-}_{\Lambda} (x, r) |\sim  \Lambda  \angles{\sqrt \beta \Lambda} \angles{  \Lambda }^{-\frac 12}.
$$
Hence  combining this estimate with \eqref{Hest} and \eqref{ph''est}  we see that the integration by parts argument goes through.

In  the stationary case, 
\begin{equation}
\label{nstcase}
 | x| \sim  \angles{\sqrt \beta \Lambda} \angles{  \Lambda }^{-\frac 12} t
\end{equation}
 we use  Lemma \ref{lm-corput}, \eqref{Hest} and \eqref{ph''est}, as in \eqref{Iest-station}, to obtain the desired estimate .
 \end{proof}

 \begin{proof}[Proof of Theorem \ref{lm-LocStr}]
We shall use the Hardy-Littlewood-Sobolev inequality which
asserts that
\begin{equation}
\label{HLSineq}
\norm{|\cdot |^{-\alpha}\ast f}_{L^a(\R)} \lesssim \ \norm{ f}_{L^b(\R)} 
\end{equation}
whenever $1 < b < a < \infty$ and $0 < \alpha< 1$ obey the scaling condition
$$
\frac1b=\frac1a +1-\alpha.
$$

First note that \eqref{Str2d} holds true for the pair $(q, r)=(\infty, 2)$ 
as this is just the energy inequality.  So we may assume $q\in(2, \infty)$.

Let $q'$ and $r'$ be the conjugates of $q$ and $r$, respectively, i.e., $q'=\frac q{q-1}$ and  $r'=\frac r{r-1}$.
 By the standard $TT^*$--argument, \eqref{Str2d} is equivalent to the estimate 
\begin{equation}
\label{TTstar}
\norm{ TT^\ast F }_{L^{q}_{t} L^{r}_{ x} (\R^{2+1}) } \lesssim \left[ \angles{ \sqrt{  \beta} \Lambda }^{-1} \angles{ \Lambda }^{\frac32} \right]^{1- \frac2r } 
\norm{ F  }_{ L^{q'}_{ t} L_x^{r'}(\R^{2+1} )},
\end{equation}
where 
\begin{equation}\label{TTastF}
\begin{split}
 TT^\ast F (x, t)&= \int_{\R^2}  \int_\R e^{i  x \cdot \xi+i(t-s)\sgn (\xi_1) {m_\beta}(  \xi)}  \rho^2_\Lambda (\xi)   \widehat{F}( \xi, s)\, ds  d\xi
 \\
 &= \int_\R  K_{\Lambda,  t-s} \ast F( \cdot,  s) \, ds,
 \end{split}
\end{equation}
with
\begin{align*}
 K_{\Lambda ,t}(x)&= \int_{\R^d}  e^{i  x \cdot \xi+it \sgn (\xi_1) {m_\beta}(  \xi)}  \rho^2_\Lambda (\xi)   \, d\xi.
\end{align*}
Observe that $$K_{\Lambda, t} \ast g (x)= S_{m_\beta}(t)  P_\Lambda g_\Lambda (x).$$  
So it follows from \eqref{dispest} that
\begin{equation}\label{kest1}
\|K_{\Lambda, t} \ast g \|_{L_x^{\infty}(\R^2)} \lesssim   \angles{ \sqrt{  \beta} \Lambda }^{-1} \angles{ \Lambda }^{\frac32}  |t|^{-1}   \|g\|_{L_x^{1}(\R^2)}.
\end{equation}
On the other hand, we have by Plancherel 
\begin{equation}\label{kest2}
\|K_{\Lambda, t} \ast g \|_{L_x^{2}(\R^2)} \lesssim    \|g\|_{L_x^{2}(\R^2)}.
\end{equation}
So interpolation between \eqref{kest1} and \eqref{kest2} yields
\begin{equation}\label{kest3}
\|K_{\Lambda, t} \ast g \|_{L_x^{r}(\R^2)} \lesssim  \left[ \angles{ \sqrt{  \beta} \Lambda }^{-1} \angles{ \Lambda }^{\frac32} \right]^{1- \frac2r }  |t|^{-\left(1-\frac2r\right)} \|g\|_{L_x^{r'}(\R^2)}.
\end{equation}
 for all $  r \in[2, \infty].$

Applying Minkowski's inequality to \eqref{TTastF}, and then  \eqref{kest3} and  \eqref{HLSineq}
with $(a, b)=(q , q' )$ and $\alpha= 1-2/r$,
 we obtain
 \begin{align*}
\norm{TT^\ast F }_{L^{q}_{t} L^{r}_{ x} (\R^{2+1})}
&\le \norm{   \int_\R \norm{ K_{\Lambda, t-s,} \ast
   F(s, \cdot) }_{L_x^r (\R^2)}  \, ds}_{L^{q}_t(\R)}
  \\
 &\lesssim   \left[ \angles{ \sqrt{  \beta} \Lambda }^{-1} \angles{ \Lambda }^{\frac32} \right]^{1- \frac2r }  \norm{  \int_\R  |t-s|^{-\left(1-\frac2r\right)}
  \norm{ F(s, \cdot) }_{ L_x^{r'}(\R^2)}  \, ds }_{L_t^{q}(\R)}
   \\
 &\lesssim \left[ \angles{ \sqrt{  \beta} \Lambda }^{-1} \angles{ \Lambda }^{\frac32} \right]^{1- \frac2r }  \norm{  
  \norm{ F }_{L_x^{r'}(\R^2) }  }_{L^{q'}_{ t} (\R)}
    \\
 &= \left[ \angles{ \sqrt{  \beta} \Lambda }^{-1} \angles{ \Lambda }^{\frac32} \right]^{1- \frac2r } 
  \norm{ F  }_{ L^{q'}_{ t} L_x^{r'}(\R^{2+1})} \, ,
\end{align*}
which is the desired estimate \eqref{TTstar}.

\end{proof}

\section{Proof of Theorem \ref{WP:theorem}}

We recall that $\epsilon=\beta=1$ in this section.

\subsection{Refined Strichartz estimate} 
As in \cite{BGT,KoTz,KeKo,Ke,LPS1,LPS2,HLRRW} the main ingredient in our analysis is a refined Strichartz estimates for solutions of the non-homogeneous linear equation
\begin{equation}\label{lp2}
\partial_tw+ L_{1 , 1} (D) \left(   1+  \frac{D_2^2}{D_1^2}\right)^\frac12 \partial_{x_1} w=F.
\end{equation}

\begin{lemma}\label{ref:Str:lemma}
Let $s>\frac74$ and $0<T$. Suppose that $w$ is a solution of the linear problem \eqref{lp2}.
Then, 
\begin{equation}\label{refstrichartz}
\|\nabla P_{> 1}w\|_{L^2_TL^{\infty}_{x}}\lesssim T^{\frac12}\| J^{s}w\|_{L^{\infty}_TL^2_x}+
\|J^{s-1}F\|_{L^2_{T,x}} \, .
\end{equation}
\end{lemma}

\begin{proof} To establish the estimate \eqref{refstrichartz} we follow the argument in \cite{KoTz}. 

Let $w_{\Lambda}:=P_{\Lambda}w$ and $F_{\Lambda}:=P_{\Lambda}F$. It is enough to prove that for any dyadic number $\Lambda>1$ and for any small real number $\theta>0$, 
\begin{equation} \label{refstrichartz.1}
\|\nabla w_{\Lambda}\|_{L^2_TL^{\infty}_{x}}\lesssim T^{\frac12}\Lambda^{\frac74+\theta}\| w_{\Lambda}\|_{L^{\infty}_TL^2_x}+
\Lambda^{\frac34+\theta}\|F_{\Lambda}\|_{L^2_{T,x}} \, .
\end{equation}
Indeed, then we would have, by choosing $\theta>0$ such that $s>\theta+\frac74$ and using Cauchy-Schwarz in $\Lambda$, 
\begin{equation*} 
\|\nabla P_{> 1}w\|_{L^2_TL^{\infty}_{x}} \le \sum_{\Lambda >1}\|\nabla w_{\Lambda}\|_{L^2_TL^{\infty}_{x}}
\lesssim T^{\frac12}\left( \sum_{\Lambda>1}\Lambda^{2s} \| w_{\Lambda}\|_{L^{\infty}_TL^2_x}^2 \right)^{\frac12}+
\left( \sum_{\Lambda>1}\Lambda^{2(s-1)} \| F_{\Lambda}\|_{L^2_{T,x}}^2 \right)^{\frac12} \, ,
\end{equation*}
which implies \eqref{refstrichartz}.

Now we prove estimate \eqref{refstrichartz.1}. To do so, we split the interval $[0,T]$ into small intervals  $I_j$ of size $\Lambda^{-1}$. In other words, we have $[0,T]=\underset{j \in J}{\cup} I_j$, where
$I_j=[a_j, b_j]$, $|I_j| \sim \Lambda^{-1}$ and $\# J \sim \Lambda T$. Observe from Bernstein's inequality that 
\begin{equation*} 
\|\nabla w_{\Lambda}\|_{L^2_TL^{\infty}_{x}} \lesssim \Lambda^{1+\frac2r}\|w_{\Lambda}\|_{L^2_TL^r_{x}} \, ,
\end{equation*}
for any $2<r<\infty$. Thus it follows applying H\"older's inequality in time that 
\begin{equation*} 
\|\nabla w_{\Lambda}\|_{L^2_TL^{\infty}_{x}} \lesssim \Lambda^{1+\frac2r} \left( \sum_{j}\|w_{\Lambda}\|_{L^2_{I_j}L^{r}_{x}}^2 \right)^{\frac12} \lesssim \Lambda^{1+\frac1r} \left( \sum_{j}\|w_{\Lambda}\|_{L^q_{I_j}L^{r}_{x}}^2 \right)^{\frac12} , 
\end{equation*} 
where $(q,r)$ is an admissible pair satisfying condition \eqref{admissible}.

Next, employing the Duhamel formula of \eqref{lp2} in each $I_j$ and recalling the definition of $S_{m_1}$ in \eqref{lin:KPFD}, we have for $t\in I_j$,
\begin{equation*}\label{rs2}
w_{\Lambda}(t)= S_{m_1}(t-a_j) w_{\Lambda}(a_j)+\int_{a_j}^t S_{m_1}(t-t')F_{\Lambda}(t') dt' \, ,
\end{equation*}
so that $\|\nabla w_{\Lambda}\|_{L^2_TL^{\infty}_{x}}$ is bounded by
\begin{equation*} 
 \Lambda^{1+\frac1r} \left( \sum_{j}\|S_{m_1}(t-a_j) w_{\Lambda}(a_j)\|_{L^q_{I_j}L^{r}_{x}}^2 +\sum_{j} \left(\int_{I_j}\|S_{m_1}(t-t')F_{\Lambda}(t')\|_{L^q_{I_j}L^{r}_{x}}dt'\right)^2\right)^{\frac12} .
\end{equation*} 

Thus it follows from Theorem \ref{lm-LocStr} that
\begin{align*}
\|\nabla w_{\Lambda}\|_{L^2_TL^{\infty}_{x}} &
\lesssim \ \Lambda^{\frac54+\frac1{2r}}\, \left(\underset{j}{\sum} \|w_{\Lambda}(a_j)\|_{L^2_{x}}^2+  \underset{j}{\sum} \left(\int_{I_j}\|F_{\Lambda}(t')\|_{L^2_x}dt'\right)^2\right)^{\frac12} \\ 
& \lesssim  T^{\frac12}\Lambda^{\frac74+\frac1{2r}}\| w_{\Lambda}\|_{L^{\infty}_TL^2_x}+\Lambda^{\frac34+\frac1{2r}}\| F_{\Lambda}\|_{L^2_{T,x}} \, ,
 \end{align*}
which implies \eqref{refstrichartz.1}  by choosing $2<r<\infty$ such that $\frac1{2r}<\theta$. 
\end{proof}

\subsection{Energy estimate}

We begin by deriving a classical energy estimate on smooth solutions of \eqref{FDKP}.
\begin{lemma} \label{energy_estimate}
Let $s>1$ and $T>0$. There exists $c_{1,s}>0$ such that for any smooth solution of \eqref{FDKP}, we have
\begin{equation} \label{energy_estimate.1}
\|u\|_{L^{\infty}_TH^s_x}^2 \le \|u(0)\|_{H^s}^2+c_{1,s}\|\nabla u\|_{L^1_TL^{\infty}_x}\|u\|_{L^{\infty}_TH^s_x}^2 \, .
\end{equation}
\end{lemma}

The proof relies on the Kato-Ponce commutator estimate (see \cite{KaPo}). 
\begin{lemma} \label{Kato-Ponce}
For $s >1$, we denote $[J^s,f]g = J^s(fg)-fJ^sg$. Then,
\begin{equation} \label{Kato-Ponce.1}
\big\| [J^s,f]g \big\|_{L^2} \lesssim \|\nabla f\|_{L^{\infty}}\|J^{s-1}g\|_{L^{2}}
+\|J^sf\|_{L^2}\|g\|_{L^{\infty}} \, .
\end{equation}
\end{lemma}

\begin{proof}[Proof of Lemma \ref{energy_estimate}]
Applying $J^s$ to \eqref{FDKP}, multiplying by $J^su$ and integrating in space leads to 
\begin{displaymath}
\frac12 \frac{d}{dt} \int_{\mathbb R^2} (J^su)^2 \, dx=-6\int_{\mathbb R^2} [J^s,u]\partial_{x_1}u \, J^su \, dx-6\int_{\mathbb R^2}uJ^s\partial_{x_1}u J^su dx \, .
\end{displaymath}
We use the Cauchy-Schwarz inequality and the commutator estimate \eqref{Kato-Ponce.1} to deal the first term on the right-hand side and integrate by parts in $x_1$ and use H\"older's inequality to deal with the second term. This implies that
\begin{equation*} 
 \frac{d}{dt}\|J^su\|_{L^2_{x}}^2 \lesssim \|\nabla u\|_{L^{\infty}_{x}} \|J^su\|_{L^2_{x}}^2 \, .
\end{equation*}
Estimate \eqref{energy_estimate.1} follows then by integrating the former estimate between $0$ and $T$ and applying H\"older's inequality in time in the nonlinear term. 
\end{proof}

Now we use the refined Strichartz estimate to control the term $\|\nabla u\|_{L^1_TL^{\infty}_x}$. 
\begin{lemma} \label{Stric_norm:lemma} Let $s>\frac74$ and $T>0$. Then, there exists $c_{2,s}>0$ such that for any solution of \eqref{FDKP}, we have
\begin{equation}\label{Stric_norm:lemma.1}
\|\nabla u\|_{L^1_TL^{\infty}_x} \le c_{2,s}T\left(1+\|u\|_{L^{\infty}_TH^s_{x}}\right)\|u\|_{L^{\infty}_TH^s_{x}} \, .
\end{equation}
\end{lemma}

\begin{proof} 
First, we deduce from the Cauchy-Schwarz inequality in time that 
\begin{equation} \label{Stric_norm:lemma.2}
\|\nabla u\|_{L^1_TL^{\infty}_x} \le T^{\frac12} \| \nabla u\|_{L^2_TL^{\infty}_x} \le T^{\frac12} \| P_{\le 1}\nabla u\|_{L^2_TL^{\infty}_x}
+T^{\frac12} \| P_{> 1}\nabla u\|_{L^2_TL^{\infty}_x} \, .
\end{equation}
The first term on the right-hand side of \eqref{Stric_norm:lemma.2} is controlled by using Bernstein's inequality,
\begin{equation}  \label{Stric_norm:lemma.3}
T^{\frac12} \| P_{\le 1}\nabla u\|_{L^2_TL^{\infty}_x} \lesssim T \|u\|_{L^{\infty}_TL^2_{x}} \, .
\end{equation}
To estimate the second term on the right-hand side of \eqref{Stric_norm:lemma.2} we use the refined Strichartz estimate \eqref{refstrichartz}. It follows that 
\begin{equation*} 
T^{\frac12} \| P_{> 1}\nabla u\|_{L^2_TL^{\infty}_x} \lesssim T\| u\|_{L^{\infty}_TH^s_x}+
T^{\frac12}\| J^{s-1}\partial_{x_1}(u^2)\|_{L^2_{T,x}} \lesssim  T\| u\|_{L^{\infty}_TH^s_x}+
T^{\frac12}\| u^2\|_{L^2_{T}H^s_x}\, .
\end{equation*}
Hence, we deduce since $H^s(\mathbb R^2)$ is a Banach algebra for $s>1$ and H\"older's inequality in time that
\begin{equation} \label{Stric_norm:lemma.4}
T^{\frac12} \| P_{> 1}\nabla u\|_{L^2_TL^{\infty}_x}  \lesssim  T\| u\|_{L^{\infty}_TH^s_x}+
T\| u\|_{L^{\infty}_{T}H^s_x}^2\, .
\end{equation}
We conclude the proof of \eqref{Stric_norm:lemma.1} gathering \eqref{Stric_norm:lemma.2}, \eqref{Stric_norm:lemma.3} and \eqref{Stric_norm:lemma.4}.
\end{proof}

\subsection{Uniqueness and $L^2$-Lipschitz bound of the flow} \label{section_uniq}
Let $u_1$ and $u_2$ be two solutions of the equation in \eqref{FDKP} in the class \eqref{WP:theorem.1} for some positive $T$ with respective initial data $u_1(\cdot,0)=\varphi_1$ and $u_2(\cdot,0)=\varphi_2$. We define the positive number $K$ by 
\begin{equation} \label{K}
K=\max\big\{\|\nabla u_1\|_{L^1_TL^{\infty}_{x}} , \|\nabla u_2\|_{L^1_TL^{\infty}_{x}} \big\} \, .
\end{equation}
We set $v=u_1-u_2$. Then $v$ satisfies 
\begin{equation} \label{fKPdiff}
\partial_tv+ L_{1 , 1} (D) \left(   1+ \epsilon \frac{D_2^2}{D_1^2}\right)^\frac12 \partial_{x_1} v+3\partial_{x_1}\big((u_1+u_2)v\big)=0 \, ,
\end{equation}
with initial datum $v(\cdot,0)=\varphi_1-\varphi_2$.

We want to estimate $v$ in $L^2(\mathbb R^2)$. We multiply \eqref{fKPdiff} by $v$, integrate in space and integrate by parts in $x_1$ to deduce that
\begin{displaymath}
\frac12\frac{d}{dt}\int_{\mathbb R^2} v^2 \, dx=-3 \int_{\mathbb R^2} \partial_{x_1}\big((u_1+u_2)v\big)v \, dx=
-\frac32\int_{\mathbb R^2}\partial_{x_1}(u_1+u_2)v^2\, dx \, .
\end{displaymath}
This implies from H\"older's inequality that 
\begin{displaymath}
\frac{d}{dt}\|v\|_{L^2_{x}}^2 \lesssim \big(\|\partial_{x_1}u_1\|_{L^{\infty}_{x}}+ \|\partial_{x_1}u_2\|_{L^{\infty}_{x}}\big)\|v\|_{L^2_{x}}^2 \, .
\end{displaymath}
Therefore, it follows from Gronwall's inequality that 
\begin{equation} \label{diffL2}
\sup_{t \in [0,T]}\|v(\cdot,t)\|_{L^2_{x}}=\sup_{t \in [0,T]}\|u_1(\cdot,t)-u_2(\cdot,t)\|_{L^2_{x}} \le e^{cK}\|\varphi_1-\varphi_2\|_{L^2} \, .
\end{equation}
Estimate \eqref{diffL2} provides the uniqueness result in Theorem \ref{WP:theorem} by choosing $\varphi_1=\varphi_2=u_0$.

\subsection{\textit{A priori} estimate}
When $s > 2$, Theorem \ref{WP:theorem} follows from a standard parabolic regularization method. The argument
also yields a blow-up criterion. 
\begin{proposition} 
Let $s>2$. Then, for any $u_0 \in H^s(\mathbb R^2)$, there exists a positive time $T(\|u_0\|_{H^s})$ and a unique maximal solution $u$ of \eqref{FDKP} in $C^0([0,T^{\star}) : H^s(\mathbb R^2))$ with $T^{\star}>T(\|u_0\|_{H^s})$.
Moreover, if the maximal time of existence $T^{\star}$ is finite, then
\begin{equation*}
\lim_{t \nearrow T^{\star}} \|u(t) \|_{H^s}=+\infty
\end{equation*}
and the flow map $u_0 \mapsto u(t)$ is continuous from $H^s(\mathbb R^2)$ to $H^s(\mathbb R^2)$. 
\end{proposition}

Let $u_0 \in H^{\infty}(\mathbb R^2)$. From the above result, there exists a solution $u \in C([0,T^{\star}) : H^{\infty}(\mathbb R^2))$ to \eqref{FDKP}, where $T^{\star}$ is the maximal time of existence of $u$ satisfying $T^{\star} \ge T(\|u_0\|_{H^s})$ and we have the blow-up alternative
\begin{equation} \label{blow-up_alt}
\lim_{t \nearrow  T^{\star}}\|u(t)\|_{H^3_{x}}=+\infty \quad \text{if} \quad T^{\star}<+\infty \, .
\end{equation}

Then, by using a bootstrap argument, we prove that the solution satisfies a suitable \textit{a priori} estimate on positive time interval depending only of the $H^s$ norm the initial datum. 
\begin{lemma} \label{a priori}
Let $\frac74 < s \le 2$. There exist $K_s>0$ and $A_s>0$  such that $T^{\star} > (A_s\|u_0\|_{H^s}+1)^{-2}$,
\begin{equation} \label{aprioest}
\|u\|_{L^{\infty}_TH^s} \le 2\|u_0\|_{H^s_x} \hskip7pt \text{and} \hskip7pt  \|\nabla u\|_{L^1_TL^{\infty}_x} \le K_s  \hskip7pt  \text{with}  \hskip7pt  T=(A_s\|u_0\|_{H^s}+1)^{-2}\, .
\end{equation}
\end{lemma}

\begin{proof} For $\frac74<s \le 2$, let us define
\begin{displaymath}
T_0:=\sup\Big\{ T \in (0,T^{\star}) \ : \ \|u\|_{L^{\infty}_TH^s_x} \le 2 \|u_0\|_{H^s_x} \Big\} \, .
\end{displaymath}
Note that the above set is nonempty since $u \in C([0,T^{\star}) : H^{\infty}(\mathbb R^2))$, so that $T_0$ is well-defined. We argue by contradiction assuming that $0<T_0<(A_s\|u_0\|_{H^s}+1)^{-2} \le 1$ for  $A_s=8(1+c_{1,s}+c_{1,3})(1+c_{2,s})$ (where $c_{1,s}$ and $c_{2,s}$ are respectively defined in  Lemmas \ref{energy_estimate} and \ref{Stric_norm:lemma}).

Let $0<T_1<T_0$. We have from the definition of $T_0$ that  $\|u\|_{L^{\infty}_{T_1}H^s}^2 \le 4 \|u_0\|_{H^s}^2$. Then estimate \eqref{Stric_norm:lemma.1} yields 
\begin{displaymath}
\|\nabla u\|_{L^1_{T_1}L^{\infty}_x} \le 2c_{2,s}T_1(1+2\|u_0\|_{H^s})\|u_0\|_{H^s} \le \frac1{4(1+c_{1,s}+c_{1,3})} \, .
\end{displaymath}
Thus, we deduce by using the energy estimate \eqref{energy_estimate.1} with $s=3$ that
\begin{displaymath}
\|u\|_{L^{\infty}_{T_1}H^3_{x}}^2 \le \frac43 \|u_0\|_{H^3}^2, \quad \forall \,  0<T_1<T_0 \, .
\end{displaymath}
This implies in view of the blow-up alternative \eqref{blow-up_alt} that $T_0<T^{\star}$.

Now, the energy estimate \eqref{energy_estimate.1} at the level $s$ yields $\|u\|_{L^{\infty}_{T_0}H^s_x}^2 \le \frac43 \|u_0\|_{H^s}^2$, so that by continuity, $\|u\|_{L^{\infty}_{T_2}H^s}^2 \le \frac53\|u_0\|_{H^s}^2$ for some $T_0<T_2<T^{\star}$. This contradicts the definition of $T_0$. 

Therefore $T_0 \ge T:=(A_s\|u_0\|_{H^s}+1)^{-2}$ and we argue as above to get the bound for $\|\nabla u\|_{L^1_{T}L^{\infty}_x}$. This concludes the proof of Lemma \ref{a priori}.
\end{proof}

\subsection{Existence, persistence and continuous dependence}
Let fix $\frac74<s \le 2$ and $u_0 \in H^s(\mathbb R^2)$. We regularize the initial datum as follows.
Let $\chi$ be the cut-off function defined in \eqref{eta}, then we define
\begin{equation*}
u_{0,n}= P_{\le n}u_0=\left(\chi({|\xi|}/{n}) \widehat{u}_0(\xi)\right)^{\vee} \, ,
\end{equation*}
for any $n \in \mathbb N$, $n \ge 1$.

Then, the following estimates are well-known (see for example Lemma 5.4 in \cite{LPS1}).
\begin{lemma} \label{BSreg}
\begin{itemize}
\item[(i)]
Let $\sigma \ge 0$ and $n \ge 1$. Then,
\begin{equation} \label{BSreg.1}
\|u_{0,n}\|_{H^{s+\sigma}} \lesssim n^{\sigma} \|u_0\|_{H^s} \, , 
\end{equation}

\item[(ii)]
Let $0 \le \sigma \le s$ and $m\ge n\ge 1$. Then, 
\begin{equation} \label{BSreg.4}
\|u_{0,n}-u_{0,m}\|_{H^{s-\sigma}} \underset{n \to +\infty}{=}o(n^{-\sigma}) 
\end{equation}
\end{itemize}
\end{lemma}

Now, for each $n \in \mathbb N$, $n \ge 1$, we consider the solution $u_n$ emanating from $u_{0,n}$ defined on their maximal time interval $[0,T^{\star}_n)$. In other words, $u_n$ is a solution to the Cauchy problem   
\begin{equation}\label{FDKPun}
\begin{cases}
\partial_tu_n+ L_{1 , 1} (D) \left( 1+ \frac{D_2^2}{D_1^2}\right)^\frac12 \partial_{x_1} u_n+ 3 \partial_{x_1}( u_n^2)=0,\;\;\;x\in\mathbb R^2, \ 0<t<T^{\star}_n, \\
u_n(x,0)=u_{0,n}(x)=P_{\le n}u_0(x), \quad x \in \mathbb R^2 \, .
\end{cases}
\end{equation}
From Lemmas \ref{a priori} and \ref{BSreg} (i), there exists a positive time 
\begin{equation} \label{defT}
T=(A_s\|u_0\|_{H^s}+1)^{-2} \, ,
\end{equation} 
(where $A_s$ is a positive constant), independent of $n$, such that $u_n \in C([0,T] : H^{\infty}(\mathbb R^2))$ is defined on the time interval $[0,T]$ and satisfies 
\begin{equation} \label{existence.1}
\|u_n\|_{L^{\infty}_TH^s_x} \le 2\|u_0\|_{H^s} 
\end{equation}
and 
\begin{equation} \label{existence.2}
K:=\sup_{n \ge 1}\big\{ \|\nabla u_n\|_{L^1_TL^{\infty}_{x}} \big\} <+\infty\, .
\end{equation}

Let $m \ge n \ge 1$. We set $v_{n,m} := u_n-u_m$. Then, $v_{n,m}$ satisfies
\begin{equation} \label{fKPdiffn}
\partial_tv_{n,m}+L_{1 , 1} (D) \left( 1+ \frac{D_2^2}{D_1^2}\right)^\frac12 \partial_{x_1}v_{n,m}+3\partial_x\big((u_n+u_m)v_{n,m}\big)=0 \, ,
\end{equation}
with initial datum $v_{n,m}(\cdot,0)=u_{0,n}-u_{0,m}$. 

Arguing as in Subsection \ref{section_uniq}, we see from Gronwall's inequality and \eqref{BSreg.4} with $\sigma=s$ that
\begin{equation} \label{existence.3}
\| v_{n,m} \|_{L^{\infty}_TL^2_{x}} \le e^{cK}\| u_{0,n}-u_{0,m} \|_{L^2} \underset{n \to +\infty}{=} o(n^{-s})
\end{equation}
which implies interpolating with \eqref{existence.1} that 
\begin{equation} \label{existence.4} 
\| v_{n,m} \|_{L^{\infty}_TH^{\sigma}_{x}} \le \|v_{n,m}\|_{L^{\infty}_TH^s_{x}}^{\frac{\sigma}s} \|v_{n,m}\|_{L^{\infty}_TL^2_{x}}^{1-\frac{\sigma}s}\underset{n \to +\infty}{=}o(n^{-(s-\sigma)}) \, ,
\end{equation}
for all $0 \le \sigma <s$.

Therefore, we deduce  that $\{u_n \}$ is a Cauchy sequence in $L^{\infty}([0,T] : H^{\sigma}(\mathbb R^2)$, for any $0 \le \sigma<s$. Hence, it is not difficult to verify passing to the limit as $n \to +\infty$ that $u =\lim_{n \to +\infty}u_n$ is a weak solution to \eqref{FDKP} in the class $C([0,T] : H^{\sigma}(\mathbb R^2)) $, for any $0 \le \sigma<s$. 

Finally, the proof that $u$ belongs to the class \eqref{WP:theorem.1} and of the continuous dependence of the flow follows from the Bona-Smith argument \cite{BS}. Since it is a classical argument, we skip the proof and refer the readers to \cite{LPS1,HLRRW} for more details in this setting.


\section{Appendix}

In this appendix, we derive some useful estimates on the first and second order derivatives of the function $m_{\beta}$ for $\beta=0,1$ defined in \eqref{def:m_beta}, which is an adaptation of the corresponding estimates for $m_0$ derived recently by the last two authors in \cite{DST}.

\begin{lemma}\label{lm-mest}
Let $\beta \in \{0,  1\}$. Then for all $r>0$ we have 
\begin{align}
\label{m'-est} 
0<m_{\beta}'(r)& \sim \angles{ \sqrt{  \beta} r} \angles{r}^{-1/2}
\\
\label{m''-est} 
|m_{\beta} ''(r)| &\sim   r  \angles{ \sqrt{ \beta }  r }\angles{ r}^{-5/2}
\end{align}

\end{lemma}

\begin{proof}
Let
$$
T(r)=\tanh(r), \quad S(r)=\text{sech}(r), \quad  K(r)= \sqrt{T(r)/r}.
$$ 
Thn
$$m_\beta(r)=r \angles{  \sqrt{ \beta} r}    K(r).$$ 

First we prove \eqref{m'-est}. Since
\begin{align*}
K'&= \frac{ r S^2-T}{2 r^2 K}=\frac 1{2 r } \left(  K^{-1}S^2-K\right)
\end{align*}
we have 
\begin{align*}
m'_\beta &=  \angles{ \sqrt{ \beta} r} (K+rK' )+ \beta r \angles{ \sqrt{ \beta} r}^{-1}(rK)
\\
&=\frac12 \angles{ \sqrt{ \beta} r} \left(K+ K^{-1} S^2\right)  +  \beta r^2 \angles{ \sqrt{ \beta} r}^{-1} K
\\
&\sim \angles{\sqrt{ \beta} r} \angles{r}^{-1/2},
\end{align*}
where in the last line we used the fact that
\begin{equation}\label{KS-Est}
K(r)\sim  \angles{r}^{-1/2} \quad \text{and} \quad  S(r)
\sim e^{-r}.
\end{equation}

Next we prove \eqref{m''-est}. 
We have 
\begin{align*}
m''_\beta= \left(2K'+ rK''\right) \angles{\sqrt{ \beta} r} + 2 \left(K+ rK'\right) \angles{\sqrt{ \beta} r}' +r  \angles{\sqrt{ \beta} r}'' K.
\end{align*}
Now we can write
\begin{align*}
K''&= -\frac{ T S^2}{ r K}- \frac{ \left( r S^2-T \right)}{ r^3K}- \frac{ \left( r S^2-T \right)^2  }{4 r^4 K^3}
\end{align*}
which in turn implies
\begin{align*}
 2K'+ rK''
&=-\frac{ TS^2}{ K}- \frac{ \left( r S^2-T \right)^2  }{4 r^3 K^3}
\\
&=-\frac {rK^{-3}}{4} \left[    4K^4S^2+   \left( (K^2-S^2)/r\right)^2 \right].
\end{align*}
We write
\begin{align*}
(K^2-S^2)/r=  ES^2,
\end{align*}
where
$$
E(r)=\frac{e^{2r}- e^{-2r}-4r} {4r^2}.
$$
We estimate $
E(r)$ as follows:
If $ 0<r< 1$ we write
\begin{align*}
E(r)= \frac1{2r^2} \int_0^r  \left(e^{2s}+e^{-2s}-2\right) \, ds 
&= \frac2{r^2}  \int_0^r  \int_0^s \int_0^t  \left(e^{2x}+e^{-2x}\right) \, dx \, dt \, ds .
\end{align*}
Then since $e^{2x}+e^{-2x}\sim 1$ for $0<x<r<1$, we have
\begin{align*}
E(r)
&\sim\frac 2{r^2}  \int_0^r  \int_0^s \int_0^t 1 \, dx \, dt \, ds \sim r.
\end{align*}
On the other hand, if $r\ge 1$, we simply have
$$
E(r)=\frac{e^{2r}} {4r^2} [1-e^{-4r}-4r e^{-2r} ] \sim \angles{r}^{-2}e^{2r}.
$$
Therefore, 
\begin{equation} \label{E-est}
E(r)\sim 
r\angles{r}^{-3}e^{2r} \quad \text{for all}\ r>0.
\end{equation}

Now letting 
\begin{align*}
A_\beta(r)&=\angles{ \sqrt{\beta} r}^{-2}   \left[ 1+  K^{-2}S^2+ \angles{\sqrt{\beta}  r}^{-2} \right],
\\
 B(r) &= 4  S^2 +  K^{-4}E^2 S^4,
  \quad
   f_\beta(r)
= 4 \beta \frac{A_\beta(r)}{B(r) }-1
\end{align*}
we can write
\begin{align*}
m''_\beta(r)= 4^{-1}r \angles{\sqrt{\beta}  r} K   B(r)  f_\beta(r).
\end{align*}

By \eqref{KS-Est} and \eqref{E-est} we have
\begin{align*}
  B(r)
 &\sim     e^{-2r}+   \angles{r}^{2} \cdot r^2  \angles{r}^{-6} e^{4r} \cdot  e^{-4r}   
\sim  \angles{r}^{-2} ,
\end{align*}
and hence
\begin{align*}
|m_\beta''(r)| \sim r \angles{\sqrt{\beta}  r}  \angles{r}^{-\frac 32}  | f_\beta(r)|.
\end{align*}
So \eqref{m''-est} reduces to proving for all $r\ge 0$
\begin{equation}
\label{qest}
 | f_\beta(r)| \sim 1 \quad  \text{for} \ \   \beta \in \{0,  1\}.
\end{equation}
 Clearly, $| f_0(r)| =1$. On the other hand, it can be easily checked  that $| f_1(r)| \sim 1$ for $r\ge 0$ (see fig.1 below).

\begin{figure}[h!]
   \includegraphics[scale=.5]{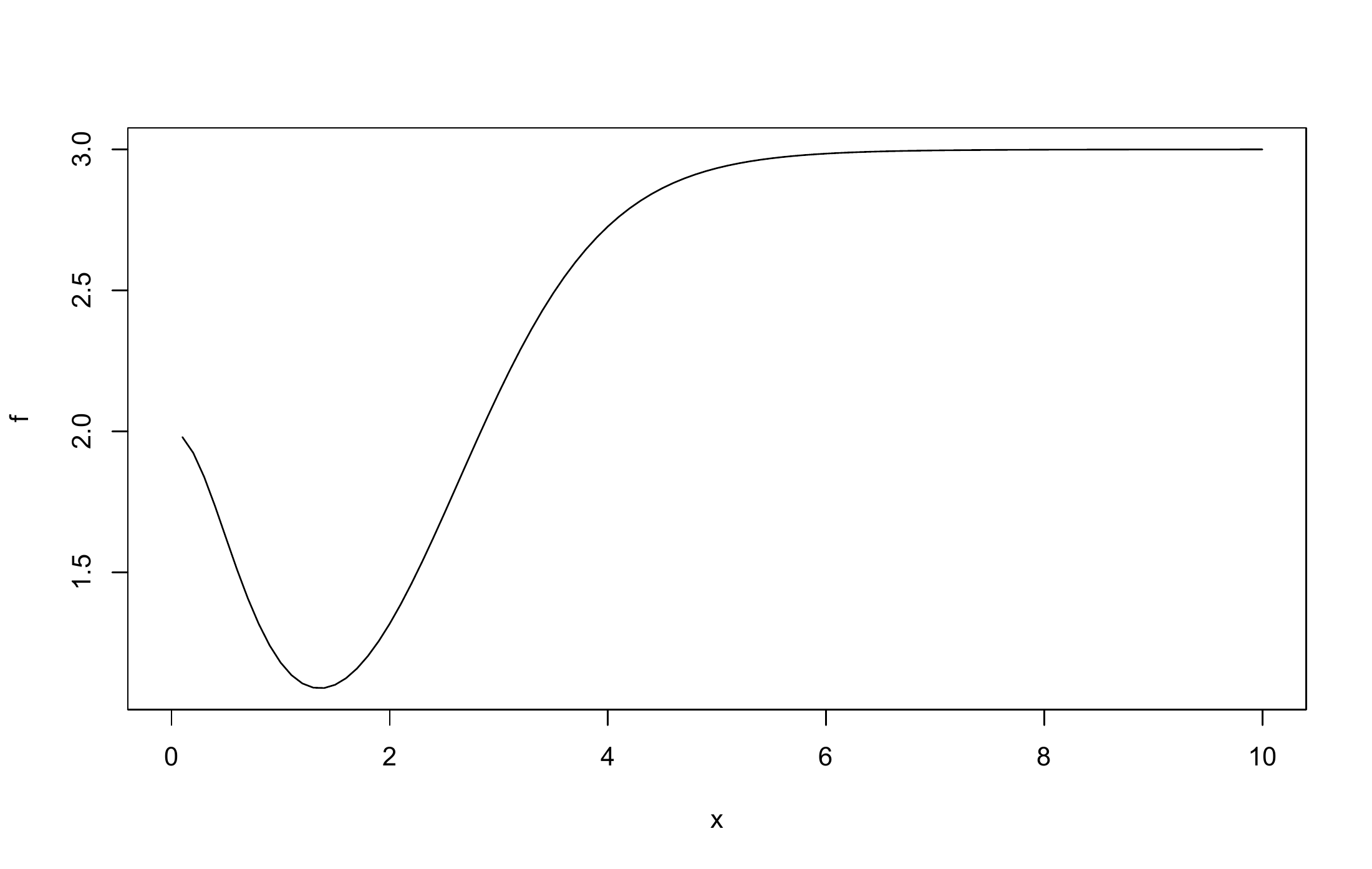}
\caption{The graph of $f_1(x)= 4 A(x) /B(x) -1$ for $x\ge 0$. It satisfies $1<f(x) \le 3$. }
\end{figure}

\end{proof}


\noindent \textbf{Acknowledgments}
D.P. was supported by the Trond Mohn Foundation grant \emph{Nonlinear dispersive equations}. S.S. and A.T. were partially supported by the Trond Mohn Foundation grant \emph{Pure mathematics in Norway}. J.-C.S. was partially supported by the ANR project ANuI (ANR-17-CE40-0035-02).

\end{document}